\documentclass[fleqn,10pt,twoside]{article}
\usepackage[headings]{espcrc1}
\usepackage{amssymb}
\usepackage{mathrsfs}
\usepackage[all]{xy} 
\newcounter{contatore}
\newenvironment{alphalist}{\begin{list}{(\alph{contatore})}{\usecounter{contatore}}}{\end{list}}
\newenvironment{proof}[1][Proof]{\noindent{\textbf{#1.}} }{\ \rule{0.5em}{0.5em}\vspace*{+2mm}}
\DeclareMathAlphabet{\mathpzc}{OT1}{pzc}{m}{it}
\title{Effective $\lambda$-models versus recursively enumerable $\lambda$-theories}

\newtheorem{thm}{Theorem}

\author{Chantal Berline\address[CNRS]{CNRS, Laboratoire PPS\\ 
        Universit\'{e} Paris 7\\ 
        2, place Jussieu (case 7014)\\ 
        75251 Paris Cedex 05, France},
        Giulio Manzonetto\textsuperscript{a}
        and
        Antonino Salibra\address[VENICE]{Universit\`{a} Ca'Foscari di Venezia\\ 
        Dipartimento di Informatica\\ 
        Via Torino 155,\\ 
        30172 Venezia, Italy}\\
        chantal.berline@pps.jussieu.fr, \{gmanzone,salibra\}@dsi.unive.it%
}
       
\runauthor{Berline, Manzonetto, Salibra}

\begin{document}

\newcommand{\lm}[1]{\mathscr{#1}} 
\newcommand{\gm}[1]{\mathscr{#1}} 
\newcommand{\ca}[1]{\mathscr{#1}} 
\newcommand{\setG}{G}
\renewcommand{\i}[1]{i_{#1}}
\newcommand{\class}[1]{\mathfrak{#1}}

\newcommand{\Tins}[1]{#1^{\cT-ins}}
\newcommand{\Zero}{\textrm{Zero}}
\newcommand{\Pred}{\textrm{Pred}}
\newcommand{\Succ}{\textrm{Succ}}
\newcommand{\BT}{\mathcal{BT}}
\newcommand{\std}[1]{\underline{#1}}
\newcommand{\cod}[1]{\ulcorner #1 \urcorner}
\newcommand{\BTle}{\sqsubseteq_\cB}
\newcommand{\BTeq}{=}
\newcommand{\BTth}{\mathcal{B_T}}
\newcommand{\BTsup}{\bigsqcup}
\newcommand{\Unsolvable}{\mathcal{U}}

\newcommand{\Visser}[1]{\mathscr{V}_{#1}}
\newcommand{\Intunder}[1]{\Lambda^o_{#1}}

\newcommand{\codepair}[2]{\langle #1,#2\rangle}
\newcommand{\codepairmix}[2]{\hspace{+2pt}\ll\hspace{-2pt} #1,#2 \hspace{-2.5pt}\gg}
\newcommand{\RE}{{\mathcal R}{\mathcal E}}
\newcommand{\CORE}{{\mathcal R}{\mathcal E}^{co}}
\newcommand{\Kleeneq}{\simeq}
\newcommand{\ecomp}{\zeta}
\newcommand{\Oinv}[1]{#1^{-}}

\newcommand{\restr}{\hspace{-4pt}\upharpoonright}
\newcommand{\partarrow}{\rightharpoonup}
\newcommand{\partinv}[1]{#1^{-1}}
\newcommand{\img}[1]{#1^{+}}
\newcommand{\inv}[1]{#1^{-}}
\newcommand{\dom}{\mathrm{dom}}
\newcommand{\rg}{\mathrm{rg}}
\renewcommand{\graph}{\mathrm{graph}}

\newcommand{\finsetenc}{\#_*}
\newcommand{\finsubset}{\subseteq_{\mathrm{f}}}
\renewcommand{\complement}[1]{#1^c}
\newcommand{\nat}{\mathbb{N}}
\renewcommand{\setminus}{-}
\renewcommand{\sup}[1]{\bigsqcup #1}
\newcommand{\cardinality}[1]{card(#1)}
\newcommand{\Kdn}[1]{\widehat{#1}}
\newcommand{\st}{:}
\newcommand{\pow}[1]{\mathscr{P}(#1)}
\newcommand{\clup}[1]{#1\hspace{-3pt}\uparrow}
\newcommand{\cldn}[1]{#1\hspace{-3pt}\downarrow}

\newcommand{\App}{Ap}
\newcommand{\Abs}{\lambda}
\newcommand{\Int}[1]{\left|  #1\right|}
\newcommand{\D}{{\mathcal D}}
\newcommand{\compel}[1]{{\bf {\cal K}}(#1)}
\newcommand{\Prime}[1]{\cP(#1)}
\newcommand{\E}{{\mathcal E}}
\newcommand{\emptyrho}{\rho_\bot}
\newcommand{\step}[2]{\varepsilon_{#1,#2}}
\newcommand{\intrho}[1]{\Int{#1}_\rho}
\newcommand{\Calgebra}{\cC}

\newcommand{\Gmin}{\gm{G}_{min}}
\newcommand{\Tmin}{\cT^{min}}
\newcommand{\Omin}{\cT^{min}_{\sqle}}
\newcommand{\Th}[1]{Th(#1)}
\newcommand{\Thle}[1]{Th_{\sqle}(#1)}

\renewcommand{\j}[1]{j_{#1}}
\newcommand{\subpair}{\sqle}
\newcommand{\suppair}{\sqsupseteq}
\newcommand{\gmcompl}[1]{\gm{G}_{#1}}
\newcommand{\ppcompl}[1]{{\overline{#1}}}
\newcommand{\funcompl}[1]{{i_{\bar{#1}}}}
\newcommand{\setcompl}[1]{{\overline{#1}}}

\newcommand{\retract}{\vartriangleleft}

\newcommand{\EffGmin}{\cG_{min}^{eff}}
\newcommand{\Intcomp}[1]{\left|  #1\right|^{comp}}
\newcommand{\domre}[1]{#1^{r.e.}}
\newcommand{\cocore}[1]{\mathscr{P}^{core}(\domre{#1})}
\newcommand{\domdec}[1]{#1^{dec}}

\newcommand{\vphi}{\varphi}
\newcommand{\W}{{\cal W}}
\renewcommand{\bold}[1]{{\bf #1}}
\newcommand{\imp}{\Rightarrow}
\newcommand{\sqle}{\sqsubseteq}

\newcommand{\Base}{\mathcal{O}}
\newcommand{\cl}[1]{\overline{#1}}

\newtheorem{conjecture}{Conjecture}
\newtheorem{problem}{Problem}

\newtheorem{theorem}{Theorem}[section]
\newtheorem{notation}[theorem]{Notation}
\newtheorem{corollary}[theorem]{Corollary}
\newtheorem{proposition}[theorem]{Proposition}
\newtheorem{lemma}[theorem]{Lemma}
\newtheorem{convention}[theorem]{Convention}
\newtheorem{definition}[theorem]{Definition}
\newtheorem{remark}[theorem]{Remark}
\newtheorem{claim}[theorem]{Claim}
\newtheorem{subclaim}[theorem]{Subclaim}
\newtheorem{example}[theorem]{Example}

\newcommand{\ot}{\overline{t}}
\newcommand{\os}{\overline{s}}
\newcommand{\sskV}{{\bf k}\sp {\bf V}}
\newcommand{\sssV}{{\bf s}\sp{\bf V}}
\newcommand{\glV}{{\gl^{\bf V}}}
\newcommand{\cdotV}{\cdot^{\bf V}}
\newcommand{\cterm}{combinatory term}
\newcommand{\bZd}{\mbox{\bf Zd\,}}
\newcommand{\Cr}{\mbox{\rm Cr\,}}
\newcommand{\ConA}{{\bf C}{\rm on}\hspace{.01in}{\bf A}}
\newcommand{\ConB}{{\bf C}{\rm on}\hspace{.01in}{\bf B}}
\newcommand{\ConAB}{{\bf C}{\rm on}\hspace{.01in}({\bf A}$\times${\bf B})}
\newcommand{\lvar}{$\gl$-variable}
\newcommand{\bI}{{\bf I }}
\newcommand{\TolA}{{\bf T}{\rm ol}\hspace{.01in}{\bf A}}
\newcommand{\FRA}{FR\hspace{.01in}{\bf A}}
\newcommand{\bnew}{\marginpar{\mbox{$\top$}}}
\newcommand{\enew}{\marginpar{\mbox{$\bot$}}}
\newcommand{\snew}{\marginpar{\mbox{$\leftarrow$}}}
\newcommand{\fnew}{\marginpar{\mbox{$\downarrow$}}}
\newcommand{\glxA}{\gl x^{\bA}}
\newcommand{\Zd}{\mbox{\rm Zd\,}}
%


\newtheorem{claimsub}{Claim}[subsection]
\newtheorem{theoremsub}{Theorem}[subsection]
\newtheorem{corollarysub}{Corollary}[subsection]
\newtheorem{definitionsub}{Definition}[subsection]
\newtheorem{propositionsub}{Proposition}[subsection]
\newtheorem{lemmasub}{Lemma}[subsection]
\newtheorem{examplesub}{Example}[subsection]

\newtheorem{claimsubsub}{Claim}[subsubsection]
\newtheorem{theoremsubsub}{Theorem}[subsubsection]
\newtheorem{corollarysubsub}{Corollary}[subsubsection]
\newtheorem{definitionsubsub}{Definition}[subsubsection]
\newtheorem{propositionsubsub}{Proposition}[subsubsection]
\newtheorem{lemmasubsub}{Lemma}[subsubsection]
\newtheorem{examplesubsub}{Example}[subsubsection]

\newcommand{\Over}[1]{\raisebox{-1.6mm}{{\scriptsize #1}}}
\newcommand{\SubOver}[1]{\raisebox{-0.6mm}{{\tiny #1}}}

\newcommand{\FOL}{\mbox{$L_{\omega,\omega}$}}

\newcommand{\Or}{\mbox{$\vee$}}
\newcommand{\Neg}{\mbox{$\neg$}}
\newcommand{\noteq}{\mbox{$\neq$}}
\newcommand{\fiproof}{\mbox{$(\Rightarrow)$}}
\newcommand{\ifproof}{\mbox{$(\Leftarrow)$}}
\newcommand{\of}{\mbox{$\in$}}
\newcommand{\notof}{\mbox{$\not\in$}}
\newcommand{\union}{\mbox{$\cup$}}
\newcommand{\intersection}{\mbox{$\cap$}}
\newcommand{\except}{\mbox{$-$}}
\newcommand{\cardof}[1]{\mbox{$|$}#1\mbox{$|$}}
\newcommand{\forevery}{\mbox{$\forall$}}
\newcommand{\therexists}{\mbox{$\exists$}}
\newcommand{\within}{\mbox{$\subseteq$}}
\newcommand{\includes}{\mbox{$\supseteq$}}
\newcommand{\product}{\mbox{$\times$}}
\newcommand{\Power}{\mbox{$\wp$}}
\newcommand{\Powernonempty}{\Power\mbox{$_{+}$}}
\newcommand{\yields}{\mbox{$\mapsto$}}
\newcommand{\comp}{\mbox{$\circ$}}
\newcommand{\emptys}{\mbox{$\emptyset$}}
\newcommand{\emptypres}{\emptys}
\newcommand{\nonabove}{\mbox{$\preceq$}}
\newcommand{\fullynonabove}{\mbox{$\sqsubseteq$}}
\newcommand{\classicallynonabove}{\mbox{$\leq$}}
\newcommand{\carrierof}[1]{\mbox{$|$}#1\mbox{$|$}}
\newcommand{\setof}[1]{\mbox{$\{$}#1\mbox{$\}$}}
\newcommand{\suchthat}{\mbox{$|$}}
\newcommand{\unionall}[3]{\mbox{$\bigcup_{#1\in#2}$}#3}
\newcommand{\widunionall}[3]{\parbox[t]{#3}
                         {\centering\mbox{$\bigcup$}
                          \\ \protect\vspace*{-1.5mm}
                          {\scriptsize #1\of#2}}}
\newcommand{\intsnall}[3]{\mbox{$\bigcap_{#1\in#2}$}#3}
\newcommand{\widintsnall}[3]{\parbox[t]{#3}
                         {\centering\mbox{$\bigcap$}
                          \\ \protect\vspace*{-1.5mm}
                          {\scriptsize #1\of#2}}}
\newcommand{\ifamily}[3]{\setof{#1\mbox{$_{#2}$}}\mbox{$_{#2\in#3}$}}
\newcommand{\indx}{{\it j}}
\newcommand{\anindx}{{\it k}}
\newcommand{\indxset}{{\it J}}
\newcommand{\iso}{\mbox{$\simeq$}}
\newcommand{\lweq}{\mbox{$\leq$}}
\newcommand{\greq}{\mbox{$\geq$}}
\newcommand{\grth}{\mbox{$>$}}
\newcommand{\eqdef}{\mbox{$\stackrel{\it def}{=}$}}
\newcommand{\inverseof}[1]{#1\mbox{$^{-1}$}}
\newcommand{\minverseof}[1]{#1^{-1}}

\newcommand{\seqof}[3]{#1\mbox{$_{1}$}#2\ldots#2#1\mbox{$_{#3}$}}

\newcommand{\seqSubOverof}[3]{#1\raisebox{-2.2mm}
                                     {{\tiny 1}}#2\ldots#2#1\raisebox{-2.2mm}
                                                                 {{\tiny #3}}}

\newcommand{\rpsreq}{(\dag)}
\newcommand{\epsreq}{(\ddag)}
\newcommand{\satinv}{(\S)}
\newcommand{\adequacy}{(\S\S)}
\newcommand{\fulladequacy}{(\S\S\S)}
\newcommand{\finitarity}{(\pounds)}

\newcommand{\upar}{\uparrow}
\newcommand{\Upar}{\Uparrow}
\newcommand{\doar}{\downarrow}
\newcommand{\Doar}{\Downarrow}

\newcommand{\riar}{\rightarrow}
\newcommand{\lear}{\leftarrow}
\newcommand{\Riar}{\Rightarrow}
\newcommand{\Lear}{\Leftarrow}
\newcommand{\leriar}{\leftrightarrow}
\newcommand{\Leriar}{\Leftrightarrow}

\newcommand{\uC}{\underline {\bf C}}
\newcommand{\uD}{\underline {\bf D}}
\newcommand{\uE}{\underline E}
\newcommand{\ugS}{\underline\Sigma}
\newcommand{\uunion}{\underline\cup}
\newcommand{\ubigunion}{\underline\bigcup}
\newcommand{\unE}{\underline E}

\newcommand{\oE}{{\overline E}}
\newcommand{\oP}{{\overline P}}
\newcommand{\oV}{{\overline V}}
\newcommand{\oW}{{\overline W}}
\newcommand{\oU}{{\overline U}}
\newcommand{\oD}{{\overline D}}

\newcommand{\cA}{{\cal A}}
\newcommand{\cB}{{\cal B}}
\newcommand{\cC}{{\cal C}}
\newcommand{\cE}{{\cal E}}
\newcommand{\cF}{{\cal F}}
\newcommand{\cG}{{\cal G}}
\newcommand{\cH}{{\cal H}}
\newcommand{\cW}{{\cal W}}
\newcommand{\cJ}{{\cal J}}
\newcommand{\cL}{{\cal L}}
\newcommand{\cM}{{\cal M}}
\newcommand{\cN}{{\cal N}}
\newcommand{\cO}{{\cal O}}
\newcommand{\cP}{{\cal P}}
\newcommand{\cQ}{{\cal Q}}
\newcommand{\cR}{{\cal R}}
\newcommand{\cS}{{\cal S}}
\newcommand{\cT}{{\cal T}}
\newcommand{\cU}{{\cal U}}
\newcommand{\cV}{{\cal V}}

\newcommand{\la}{\langle}
\newcommand{\ra}{\rangle}

\newcommand{\sm}{\leadsto}
\newcommand{\sopra}[2]{\stackrel{#1}{#2}}

\newcommand{\bb}{{\bf b}}
\newcommand{\bA}{{\bf A}}
\newcommand{\bB}{{\bf B}}
\newcommand{\bC}{{\bf C}}
\newcommand{\bD}{{\bf D}}
\newcommand{\bE}{{\bf E}}
\newcommand{\bF}{{\bf F}}
\newcommand{\bK}{{\bf K}}
\newcommand{\bM}{{\bf M}}
\newcommand{\bN}{{\bf N}}
\newcommand{\bP}{{\bf P}}
\newcommand{\bR}{{\bf R}}
\newcommand{\bS}{{\bf S}}
\newcommand{\bT}{{\bf T}}
\newcommand{\bU}{{\bf U}}
\newcommand{\bX}{{\bf X}}
\newcommand{\bRel}{{\bf Rel}}
\newcommand{\bx}{{\bf x }}

\newcommand{\bCX}{{{\bf C}[X]}}

\newcommand{\ga}{\alpha}
\newcommand{\gb}{\beta}
\newcommand{\gam}{\gamma}
\newcommand{\gd}{\delta}
\newcommand{\gep}{\varepsilon}
\newcommand{\gz}{\zeta}
\newcommand{\geta}{\eta}
\newcommand{\gth}{\vartheta}
\newcommand{\gi}{\iota}
\newcommand{\gv}{\nu}
\newcommand{\gk}{\kappa}
\newcommand{\gl}{\lambda}
\newcommand{\gn}{\nu}
\newcommand{\gx}{\xi}
\newcommand{\gp}{\pi}
\newcommand{\gr}{\rho}
\newcommand{\gs}{\sigma}
\newcommand{\gt}{\tau}
\newcommand{\gu}{\upsilon}
\newcommand{\gph}{\varphi}
\newcommand{\gch}{\chi}
\newcommand{\gps}{\psi}
\newcommand{\go}{\omega}

\newcommand{\gG}{\Gamma}
\newcommand{\gF}{\Phi}
\newcommand{\gD}{\Delta}
\newcommand{\gT}{\Theta}
\newcommand{\gP}{\Pi}
\newcommand{\gX}{\Xi}
\newcommand{\gS}{\Sigma}
\newcommand{\gO}{\Omega}
\newcommand{\gL}{\Lambda}
\newcommand\LT{\mbox{\sf LT}}
\newcommand\sfB{\mbox{\sf B}}
\newcommand\DCAI{\mbox{\sf DCA}_I}
\newcommand\RFA{\mbox{\sf RFA}}
\newcommand\RFAI{\mbox{\sf RFA}_I}
\newcommand\RFAJ{\mbox{\sf RFA}_J}
\newcommand\LAA{\mbox{\sf LAA}}
\newcommand\LAAI{\mbox{\sf LAA}_I}
\newcommand\WLAI{\mbox{\sf LAA}_I^w}
\newcommand\SWLAI{\mbox{\sf LAA}_I^{sw}}
\newcommand\LIEI{\mbox{\sf LIE}_I}
\newcommand\CLAI{\mbox{\sf LAA}_I^c}
\newcommand\CLFAI{\mbox{\sf LFA}_I^c}
\newcommand\SLAI{\mbox{\sf LAA}_I^s}
\newcommand\SSLAI{\mbox{\sf LAA}_I^{cs}}
\newcommand\OLAAI{\mbox{\sf OLAA}_I}
\newcommand\LAAIE{\mbox{\sf LAA}_I^\eta}
\newcommand\LAAJ{\mbox{\sf LAA}_J}
\newcommand\BA{\mbox{\sf BA}}
\newcommand\GSA{\mbox{\sf GSA}}
\newcommand\RGA{\mbox{\sf RGA}}
\newcommand\GA{\mbox{\sf GA}}
\newcommand\K{\mbox{\sf K}}
\newcommand\G{\mbox{\sf G}}

\newcommand\FLA{\mbox{\sf FLA}}
\newcommand\FLAI{\mbox{\sf FLA}_I}
\newcommand\FLAJ{\mbox{\sf FLA}_J}
\newcommand\LFA{\mbox{\sf LFA}}
\newcommand\LFAI{\mbox{\sf LFA}_I}
\newcommand\LFAJ{\mbox{\sf LFA}_J}
\newcommand\CA{\mbox{\sf CA}}
\newcommand\LA{\mbox{\sf LA}}
\newcommand\LM{\mbox{\sf LM}}
\newcommand{\CL}{\mbox{ \sf CL}}

\newcommand\ssk{{\bf k}}
\newcommand\sss{{\bf s}}
\newcommand\ssm{\mbox{\boldmath $m$}}
\newcommand\ssi{{\bf i}}
\newcommand\ssl{\mbox{\boldmath $l$}}
\newcommand\sso{\mbox{\boldmath $1$}}
\newcommand\sst{\mbox{\boldmath $2$}}
\newcommand\ssT{\mbox{\boldmath $T$}}
\newcommand\ssF{\mbox{\boldmath $F$}}
\newcommand\ccc{\mbox{\boldmath $c$}}

\newcommand\Con{\mbox{Con}}

\newcommand\sbA{{\mbox{\bf A}}}%
\newcommand\ssbA{{\mbox{\footnotesize\bf A}}}%
\newcommand\sbB{\mbox{\boldmath $B$}}%
\newcommand\sbC{{\mathbf C}}%
\newcommand\sbD{\mbox{\boldmath $D$}}%
\newcommand\sbE{\mbox{\boldmath $E$}}%
\newcommand\sbF{{\mbox{\bf F}}}%
\newcommand\sbG{\mbox{\boldmath $G$}}%
\newcommand\sbH{\mbox{\boldmath $H$}}%
\newcommand\sbI{\mbox{\boldmath $I$}}%
\newcommand\sbJ{\mbox{\boldmath $J$}}%
\newcommand\sbK{{\bold K}}%
\newcommand\sbL{{\bold L}}%
\newcommand\sbM{\mbox{\bf M}}%
\newcommand\sbN{\mbox{\bf N}}%
\newcommand\sbO{{\bold O}}%
\newcommand\sbP{{\bold P}}%
\newcommand\sbQ{{\bold Q}}%
\newcommand\sbR{{\bold R}}%
\newcommand\sbS{{\bold S}}%
\newcommand\sbT{{\bold T}}%
\newcommand\sbU{{\bold U}}%
\newcommand\sbV{{\bold V}}%
\newcommand\sbW{{\bold W}}%
\newcommand\sbX{{\bold X}}%
\newcommand\sbY{{\bold Y}}%
\newcommand\sbZ{{\bold Z}}%

\newcommand\oa{{\overline{a}}}
\newcommand\ob{{\overline{b}}}
\newcommand\oc{{\overline{c}}}
\newcommand\od{{\overline{d}}}
\newcommand\ox{{\overline{x}}}
\newcommand\oy{{\overline{y}}}
\newcommand\ou{{\overline{u}}}
\newcommand\ov{{\overline{v}}}
\newcommand\ow{{\overline{w}}}
\newcommand\oz{{\overline{z}}}
\newcommand\oQ{{\overline{Q}}}
\newcommand\oN{{\overline{N}}}
\newcommand\oM{{\overline{M}}}
\newcommand\oR{{\overline{R}}}
\newcommand\oX{{\overline{X}}}
\newcommand\ogx{{\overline{\gx}}}
\newcommand\oempty{{\overline{\emptyset}}}

\newcommand\twohead{{\twoheadrightarrow}}

\newcommand\up{{\uparrow}}
\newcommand\down{{\downarrow}}
\newcommand\updown{{\updownarrow}}

\newcommand{\bCI}{{\bC[I]}}
\newcommand{\bDI}{{\bD[I]}}

\newcommand\sseq{{\subseteq}}
\newcommand\arrow{{\to}}

\maketitle

\begin{abstract}\bold{Abstract.} 
A longstanding open problem is whether there exists a non-syntactical model of the untyped $\lambda$-calculus whose theory is exactly the least 
$\lambda$-theory $\lambda_\beta$.
In this paper we investigate the more general question of whether the equational/order theory of a model of the untyped  $\lambda$-calculus can 
be recursively enumerable (r.e.\ for brevity).
We introduce a notion of \emph{effective model} of  $\lambda$-calculus, which covers in particular all the models
individually introduced  in the literature.
We prove that the order theory of an effective model is never r.e.; from this it follows that its equational theory cannot be $\lambda_\beta$, $\lambda_{\beta\eta}$.
We then show that no effective model living in the stable or strongly stable semantics has an r.e. equational theory.
Concerning Scott's semantics, we investigate the class of graph models and prove that  no order theory of a graph model can be r.e., and that
there exists an effective graph model whose  equational/order  theory is the minimum among the theories of graph models.
Finally, we show that the class of graph models enjoys a kind of downwards L\"owenheim-Skolem theorem.

\end{abstract}

\smallskip

{\bf key words}: $\lambda$-calculus, effective $\lambda$-models, effectively given domains, recursively enumerable $\lambda$-theories,
graph models, L\"owenheim-Skolem theorem.
\tableofcontents

\newpage
\section{Introduction}

\subsection{Lambda-theories and lambda-models}

\bold{$\lambda$-theories} are, by definition, the equational extensions of the untyped $\lambda$-calculus which are 
closed under derivation \cite{Bare}; in other words: a $\lambda$-theory is a $\lambda$-congruence which contains
$\beta$-conversion ($\lambda_{\beta})$; \emph{extensional} $\lambda$-theories are those which contain $\beta\eta$-conversion 
($\lambda_{\beta\eta}$). $\lambda$-theories arise by syntactical or by semantic considerations. 
Indeed, a $\lambda$-theory $\cT$, may correspond to a possible operational (observational)
semantics of $\lambda$-calculus, as well as it may be induced by a model $\lm{M}$ of $\lambda$-calculus through the kernel congruence relation of
the interpretation function (then we will say that $\lm{M}$ represents $\cT$ and we will write $\Th{\lm{M}}=\cT$). 
Although researchers have, till recently, mainly focused their interest on a limited number of them, the set of $\lambda$-theories ordered by inclusion
constitutes a very rich, interesting and complex mathematical structure (see \cite{Bare,Berline00,Berline06,LusinS04}), whose cardinality is $2^{\aleph_0}$.\\

\bold{$\lambda$-models.} After the first model, found by Scott in 1969 in the category of complete lattices and Scott continuous functions, 
a large number of mathematical models for $\lambda$-calculus, arising from syntax-free constructions, have been
introduced in various categories of domains and were classified into semantics according to the nature of their 
representable functions, see e.g. \cite{Bare,Berline00,Plotkin93}. 
Scott's continuous semantics \cite{Scott72} is given in the category whose objects are complete partial orders and
morphisms are Scott continuous functions. The stable semantics (Berry \cite{Berry78}) and the strongly stable semantics 
(Bucciarelli-Ehrhard \cite{BucciarelliE91}) are refinements of the continuous semantics, introduced to approximate the notion 
of ``sequential'' Scott continuous function; finally ``weakly continuous'' semantics have been introduced, either for modeling non
determinism, or for foundational purposes. 
In each of these semantics all the models come equipped with a partial order, and some of them, called
\emph{webbed models}, are built from lower level structures called ``webs''.
The simplest class of webbed models is the class of \emph{graph models}, which was isolated in the seventies by Plotkin, 
Scott and Engeler within the continuous semantics. 
The class of graph models contains the simplest non syntactical models of $\lambda$-calculus (to begin with Engeler's model $\lm{E)}$, is itself the easiest 
describable class, and represents nevertheless $2^{\aleph_{0}}$ (non extensional) $\lambda$-theories. 
The results previously obtained for the class of graph models are surveyed in \cite{Berline06}.
Scott continuous semantics also includes the class \emph{filter models}, which were isolated at the beginning of eighties by Barendregt, Coppo and Dezani \cite{BarendregtCD83} after the introduction of intersection-type discipline at the end of seventies by Coppo and Dezani \cite{CoppoD80}. Filter models are perhaps the most established and studied semantics of $\gl$-calculus (see e.g. \cite{Ronchi82,CoppoDHL82,CoppoDZ87}).

\subsection{The problems we are interested in}
~\\
\bold{The initial problem.} The question of the existence of a non-syntactical model of $\lambda_{\beta}$
($\lambda_{\beta\eta}$) has been circulating since at least the beginning of the eighties\footnote{See Problem 22 in the list of TLCA open problems \cite{HonsellTLCA}.}, but it was only first raised in print in \cite{HonsellR92}.
This problem is still open, but generated a wealth of interesting research and results (surveyed in \cite{Berline00} and \cite{Berline06}), 
from which we only sketch below what is relevant for the present paper.\\

\bold{The first results.} In 1995 Di Gianantonio, Honsell and Plotkin succeeded to build an extensional
model having theory $\lambda_{\beta\eta}$, living in some weakly continuous semantics \cite{DiGianantonioHP95}. 
However, the construction of this model as an inverse limit starts from the term model of $\lambda_{\beta\eta}$, 
and hence involves the syntax of $\lambda$-calculus. 
Furthermore the existence of a model living in Scott's semantics itself, or in one of its two refinements,
remains completely open. 
Nevertheless, the authors also proved in \cite{DiGianantonioHP95} that the set of extensional theories representable by models living in
Scott's semantics had a least element. 
At the same time Selinger proved that if an ordered model has theory $\lambda_{\beta}$ or $\lambda_{\beta\eta}$ then
the order is discrete on the interpretations of $\lambda$-terms \cite{Selinger03}.\\

\bold{First extension: the minimality problem.} In view of the second result of \cite{DiGianantonioHP95}, it becomes natural to ask whether, 
given a (uniformly presented) class of models of $\lambda$-calculus, there is a minimum $\lambda$-theory represented in it; 
a question which was raised in \cite{Berline00}. 
In \cite{BucciarelliS04,BucciarelliS08} Bucciarelli and Salibra showed that the answer is also positive for the
class of \emph{graph models}, and that the least \emph{graph theory} (theory of a graph model) 
was different from $\lambda_{\beta}$ and of course $\lambda_{\beta\eta}$.
At the moment the problem remains open for the other classes of models.\\

\bold{Each class of models represents and omits $2^{\aleph_0}$ $\lambda$-theories.} Ten years ago, it was proved that in each of the known 
(uniformly presented) classes $\class{C}$ of models, living in any of the above mentioned semantics, and to begin with the class of graph models, it is 
possible to build $2^{\aleph_{0}}$ (webbed) models inducing pairwise distinct $\lambda$-theories \cite{Kerth98b,Kerth01}. 
More recently, it has been proved in \cite{Salibra03} that there are $2^{\aleph_{0}}$ theories which are omitted by all the $\class{C}$'s, 
among which $\aleph_{0}$ are finitely axiomatizable over $\lambda_{\beta}$.

From these results, and since there are only $\aleph_{0}$ recursively enumerable theories (\emph{r.e}. in the sequel), 
it follows that each $\class{C}$ represents $2^{\aleph_{0}}$ non r.e.\ theories and omits
$\aleph_{0}$ r.e.\ theories. 
Note also that there are only very few theories of non syntactical models which are known to admit an alternative description
(e.g. via syntactical considerations), and that all happen to coincide either with the theory $\BTth$ of B\"{o}hm trees \cite{Bare} 
or some variations of it, and hence are non r.e. 
This leads us to raise the following problem, which is a second natural generalization of the initial problem.\\

\bold{Can a non syntactical model have an r.e.\ theory?} This problem was first raised in \cite{Berline06}, where it is conjectured
that no graph model can have an r.e.\ theory. 
But we expect that this could indeed be true for all $\lambda$-models living in the continuous semantics,
or in its refinements (but of course not in its weakenings, because of \cite{DiGianantonioHP95}), and in the 
present paper we extend officially this conjecture.

\begin{conjecture}\label{conj:ScottSemantics} 
No $\lambda$-model living in Scott's continuous semantics or in one of its refinements has an r.e.\ equational theory.
\end{conjecture}

\subsection{Methodology} \label{sec:methodology}
~\\
\indent 1) \emph{Look also at order theories.} Since all the models we are interested
in are partially ordered, and since, in this case, the equational theory $\Th{\lm{M}}$ is easily 
expressible from its order theory $\Thle{\lm{M}}$ (in particular if $\Thle{\lm{M}}$ is r.e.\ then 
also $\Th{\lm{M}}$ is r.e.) we will also address the analogue problem for order theories.

2) \emph{Look at models with built-in effectivity properties}. 
There are several reasons to do so. 
First, it may seem reasonable to think that, if effective models do not even succeed to have an r.e.\ theory, then it is
unlikely that the other ones may succeed; second, because all models which have been individually studied or given as examples in the literature are effective, in our sense.
Starting from the known notion of an effective domain, we introduce an appropriate notion of an \emph{effective model of $\lambda$-calculus}
and we study the main properties of these models\footnote{
As far as we know, only Giannini and Longo \cite{GianniniL84} have introduced a
notion of an effective model; moreover, their definition is \emph{ad hoc} for
two particular models (Scott's $P_{\omega}$ and Plotkin's $T_{\omega}$) and
their results depend on the fact that these models have a very special common theory, namely $\BTth$.
}. 
Note that, in the absolute, effective models happen to be rare, since each ``uniform'' class $\class{C}$ represents
$2^{\aleph_{0}}$ theories, but contains only $\aleph_{0}$ non-isomorphic effective models! 
However, and this is a third \emph{a posteriori} reason to work with them, it happens that they can be used to prove 
properties of non effective models (Theorem \ref{Theorem graph} below is the first example we know of such a result).

3) A previous result obtained for \emph{typed $\lambda$-calculus} also justifies the above methodology. 
Indeed, it was proved in \cite{BerardiB02} that there exists a (webbed) model of Girard's system $F$, 
living in Scott's continuous semantics, whose theory is the typed version of $\lambda_{\beta\eta}$, and whose
construction does not involve the syntax of $\lambda$-calculus. 
Furthermore, this model can easily be checked to be ``effective'' in the same spirit as in
the present paper (see \cite[Appendix C]{BerardiB02} for a sketchy presentation of the model). 
Note that this model has no analogue in the stable semantics.

4) \emph{Look at the class of graph models}. 
Recall, from a remark above, that a graph model can be effective but
that most of them are not. 

5) \emph{Prove a L\"owenheim-Skolem theorem}. 
Effective webbed models are, in particular, generated by countable webs. 
A key step for attacking the general conjecture is hence to prove that the order/equational theory of any webbed model 
can be represented by a model of the same kind but having a countable web. 
We will prove this here for graph models. 

6) \emph{Mention when the results extend} to some other class of webbed models, 
and when they do not (sometimes we do not know). 
All the classes of webbed models indeed appear to be (more or less) sophisticated variations of the class of graph models.
Studying graph models illustrates the spirit of the tools we aim at developing, while keeping 
technicalities at the lowest possible level. 
We will not work out the details, since this would lead us to far, and would be teadious, with no special added interest. 
Our program is rather to search for generic tools and our first success in this direction concerns a meta-L\"owenheim-Skolem theorem whose proof will be given in a further paper.

\subsection{Main results and derived conjectures}\label{sec:Mainresandconj}
~\\
\indent \bold{I. On effective models.} 

The central technical device here is Visser's result \cite{Visser80} stating
that the complements of $\beta$-closed r.e.\ sets of $\lambda$-terms enjoy the finite intersection property (Theorem \ref{thm:hyperconnection}). 
We will be able to prove the following.

\begin{thm}
Let $\lm{M}$ be an \emph{effective} model of $\lambda$-calculus. Then:

(i) $\Thle{\lm{M}}$ is not r.e.

(ii) $\Th{\lm{M}}\neq\lambda_{\beta},\lambda_{\beta\eta}.$

(iii) If $\bot_{\lm{M}}$ is $\lambda$-definable then $\Th{\lm{M}}$ is not r.e., more generally:

(iv) If there is a $\lambda$-term $M$ such that in $\lm{M}$ there are only finitely many $\lambda$-definable elements below the interpretation of
$\lm{M}$ then $\Th{\lm{M}}$ is not r.e.
\end{thm}

Concerning the existence of a non-syntactical effective model with an r.e.\ equational theory, we are able to give a definite answer for all (effective)
stable and strongly stable models:

\begin{thm}
No effective model living in the stable or in the strongly stable semantics has an r.e.\ equational theory.
\end{thm}

This theorem solves Conjecture~\ref{conj:ScottSemantics} for these two semantics. 
Concerning Scott's semantics, the problem looks much more difficult and we concentrate on the class of graph models.\\

\bold{II. On graph models.}

\begin{thm}\label{Theorem min}
There exists an \emph{effective} graph model whose equational/order theory is the minimum graph theory.
\end{thm}

\begin{thm}\label{Theorem graph} 
If $\lm{M}$ is a graph model then $\Thle{\lm{M}}$ is not r.e.
\end{thm}

We emphasize that Theorem \ref{Theorem graph}, which happens to be a consequence of Theorem \ref{Theorem min}, 
plus the work on effective models, concerns all the graph models and not only the effective ones.
Concerning the equational theories of graph models we only give below, as
Theorem \ref{Theorem graph2}, the more flashy example of the results we will
prove in Section~\ref{subsec:r.e.graphtheories}. 
The stronger versions are however natural, and needed for covering all the traditional models 
(for example the Engeler model is covered by Theorem~\ref{Theorem graph2} below only if it  is generated from a
finite set of atoms, while it is well known that its theory is $\BTth$, independently of the number of its atoms).

\begin{thm}\label{Theorem graph2} 
If $\lm{M}$ is a graph model which is ``freely generated from a finite partial web'', then $\Th{\lm{M}}$ is not r.e.
\end{thm}

It remains open whether the minimum equational graph theory is r.e.
Hence, the following instances of Conjecture~\ref{conj:ScottSemantics}  are still open; we state them from the weaker to the stronger one.


\begin{conjecture}
The minimum equational graph theory is non r.e.
\end{conjecture}

\begin{conjecture}
All the effective graph models have non r.e.\ equational theories.
\end{conjecture}

\begin{conjecture}
All the effective models living in the continuous semantics have non r.e.\ equational theories.
\end{conjecture}


The following further theorem states that graph models with countable webs are enough for representing all graph theories.
This can be viewed as a kind of L\"{o}wenheim-Skolem Theorem for graph models (see Section~\ref{skolem} for more comments).

\begin{thm}\label{LSKgen} For any graph model $\lm{G}$ there is a graph model $\lm{G'}$ which has a countable web and the same order theory 
(and hence the same equational theory).
\end{thm}

This result answers positively Problem 12 in \cite{Berline06}.

The more general problem concerning all known classes of webbed models appeared previously as Question 3 in \cite[Sec. 6.3]{Berline00}.
We are now able to give a full positive answer to Question 3, relying on a more conceptual proof.
We will keep this development for a later work.

\bigskip

\emph{The paper is an expanded version of ``Lambda theories of effective lambda-models'' \cite{BerlineMS07}. 
Besides containing more proofs, explanations, and examples, it also contains some deeper results 
(e.g., Theorem~\ref{thm:extendedthm} and its corollaries).}

\part{Preliminaries}

\section{Generalities}\label{pre}

To keep this article as self-contained as possible, we summarize some definitions and results that we will use later on.
Concerning $\lambda$-calculus, we will generally use the notation of Barendregt's classic work \cite{Bare}. 

\subsection{Sets, functions and groups of automorphisms}\label{subsec:sets}
We will denote by $\nat$ the set of natural numbers and by $p_k$ the $k$-th prime number.
If $X$ is a set, $\pow{X}$ (resp. $X^*$) is the set of all subsets (resp. finite subsets) of $X$.
We write $X\finsubset Y$ to express that $X$ is a finite subset of $Y$.

For any function $f$ we write $\dom(f)$ for the domain of $f$, $\rg(f)$ for its range, $\graph(f)$ for its graph, 
and $f\restr_X$ for its restriction to a subset $X\subseteq \dom(f)$.
We define the \emph{image}\index{image}\label{intro:image} and the \emph{inverse image}\index{inverse image} 
of $X$ via $f$ respectively as $\img{f}(X) = \{ f(x)\st x\in X\}$ and $\inv{f}(X) = \{x\st f(x)\in X\}$.
The \emph{partial inverse}\index{partial inverse}\label{intro:partial inverse} of an injective function $f$, denoted by $\partinv{f}$, is defined by:
$\dom(\partinv{f}) = \rg(f)$ and $\partinv{f}(x) = y$ if $f(y) = x$.

Let $f,g$ be two partial functions, then: $f$ and $g$ are \emph{compatible} if $f(x)=g(x)$ for all $x\in \dom(f)\cap \dom(g)$;
$f\cap g$ denotes the function whose graph is $\graph(f)\cap \graph(g)$;
if $f,g$ are compatible, we denote by $f\cup g$ the function whose graph is $\graph(f)\cup \graph(g)$;
finally, $f(x) \Kleeneq g(y)$ abbreviates 
$f(x)$ is undefined if, and only if, $g(y)$ is undefined and, if they are both defined, $f(x)=g(y)$.

Given any mathematical structure $\cS$ having a carrier set $S$, we denote by $Aut(\cS)$
the group of all the automorphisms of $\cS$.
For all $s\in S$ the \emph{orbit} $O(s)$ with respect to $Aut(\cS)$ is defined by $O(s) = 
\{\theta(s)\st \theta\in Aut(\cS)\}$.
A structure $\cS$ is \emph{finite modulo $Aut(\cS)$} if the number of orbits of $\cS$, with
respect to $Aut(\cS)$, is finite.

\subsection{Recursion theory}\label{subsec:recth}

We write $\vphi_n:\nat\to\nat$ for the partial recursive function of index $n$ and we indicate by $\W_n$ the domain of $\vphi_n$. 
A set $E\subseteq\nat$ is \emph{recursively enumerable} (r.e.\ for short) if it is the domain of a 
partial recursive function. The complement $\complement{E}$ of an r.e.\ set $E$ is called \emph{co-r.e.}
If both $E$ and $\complement{E}$ are r.e., $E$ is called \emph{decidable}.
Note that the collection of all r.e.\ (co-r.e.) sets is closed under finite 
union and finite intersection.

We say that $\gv$ is an \emph{encoding} of a countable set $X$ if $\gv:X\to\nat$ is bijective.
A \emph{numeration} $\gamma$ is a pair $(X,\gv_X)$, such that $\gv_X:\nat\to X$ is total and onto.
Thus, the inverse of an encoding is a special case of numeration.
A set $Y\subseteq X$ is \emph{r.e.\ } (resp. \emph{co-r.e.}) with respect to $\gv_X$ if the set $\inv{\gv_X}(Y)$ is r.e.\ (resp. co-r.e.).

Given two numerations $(X,\gv_X)$ and $(Y,\gv_Y)$ we say that a partial recursive function $\vphi$ \emph{tracks}
$f:X\to Y$ with respect to $\gv_X,\gv_Y$ if the following diagram commutes:
$$
\xymatrix{
\nat \ar[rr]^\vphi \ar[d]_{\gv_X} && \nat \ar[d]^{\gv_Y} \\
X \ar[rr]_{f} && Y\\ 
}
$$
A function $f:X\to Y$ is said \emph{computable} (with respect to $\gv_X,\gv_Y$) if there exists $\vphi$ tracking $f$ with respect to $\gv_X,\gv_Y$.
Hereafter we suppose that a computable encoding $\codepair{-}{-}:\nat^2\to\nat$ for the pairs is fixed.
Moreover, we fix an encoding $\finsetenc:\nat^*\to\nat$ 
which is effective in the sense that the relations $m\in\partinv{\finsetenc}(n)$ and $m=\cardinality{\partinv{\finsetenc}(n)}$ are decidable in $(m,n)$.
Finally we set $\codepairmix{-}{-}:\nat^*\times\nat \to \nat$ defined as $\codepairmix{a}{n} = \codepair{\finsetenc(a)}{n}$.
We recall here a basic property of recursion theory which we will often use in the sequel.

\begin{remark}\label{rem:comp1} The inverse image of an r.e.\ set via a computable map is r.e.
\end{remark}

\subsection{Partial Orderings}\label{subsec:poset} 
Let $(\D,\sqle_\D)$ be a partially ordered set (poset, for short). 
When there is no ambiguity we write $\D$ instead of $(\D,\sqle_\D)$.
Two elements $u$ and $v$ of $\D$ are: 
\emph{comparable} if either $u \sqle_\D v$ or $v\sqle_\D u$; 
\emph{compatible} if they have an upper bound, i.e., there exists $z$ such that $u \sqle_\D z$ and $v\sqle_\D z$.

Let $A\subseteq \D$ be a set. $A$ is \emph{upward} (resp. \emph{downward}) \emph{closed} if $v\in A$ and $v\sqle_\D u$ 
(resp. $u\sqle_\D v$) imply $u\in A$. 
$A$ is \emph{directed} if, for all $u,v\in A$, there exists $z\in A$ such that $u\sqle_\D z$ and $v\sqle_\D z$.

A poset $\D$ is a \emph{complete partial order} (\emph{cpo}, for short) if it has a
least element (denoted by $\bot_\D$) and every directed set $A\subseteq \D$ admits a least upper bound (denoted by $\sup{A}$).
A cpo is \emph{bounded complete} if $\sup{\{u,v\}}$ exists for all compatible elements $u,v$.
An element $d\in \D$ is called \emph{compact} if for every directed $A\subseteq \D$ we have that 
$d\sqle_\D \sup{A}$ implies $d\sqle_\D v$ for some $v\in A$. 
We write $\compel{\D}$ for the collection of compact elements of $\D$.
A cpo $\D$ is \emph{algebraic} if for every $u\in \D$ the set $\{d \in\compel{\D} \st d\sqle_\D u\}$ 
is directed and $u$ is its least upper bound. 
An algebraic cpo $\D$ is called \emph{$\omega$-algebraic} when $\compel{\D}$ is countable.
A bounded complete $\omega$-algebraic cpo is called a \emph{Scott domain}.
A compact element $p\neq \bot_\D$ of a Scott domain $\D$ is \emph{prime} if, for all compatible $u,v\in \D$, we have that
$p\sqle_\D u\sqcup v$ implies $p \sqle_\D u$ or $p\sqle_\D v$. 
We denote by $\Prime{\D}$ the set of prime elements of $\D$.
A Scott domain $\D$ is \emph{prime algebraic} if for all $u\in \D$ we have $u=\sup{\{p\in\Prime{\D} \st p\sqle_\D u\}}$.

The simplest examples of prime algebraic domains are the \emph{flat domains} and the \emph{powerset domains}.
If $D$ is a set and $\bot$ an element not belonging to $D$, the \emph{flat domain} $D_\bot$ is, by definition, the poset
$(\D,\sqle_\D)$ such that $\D = D\cup\{\bot\}$ and for all $u,v\in \D$ we have $u\sqle_\D v$ if, and only if, $u = \bot$ or $u = v$.
All elements of $\D\setminus \{\bot\}$ are prime.
Concerning the full powerset domain $(\pow{D},\subseteq)$,
the compact elements are the finite subsets of $D$ and the prime elements are the singleton sets. 
We have a Scott domain when $D$ is countable.

\section{The untyped $\lambda$-calculus}\label{lambdacalculus}

\subsection{$\lambda$-terms}
The set $\Lambda$ of $\lambda$-terms over a countable set of variables is constructed as usual:
every  variable is a $\lambda$-term; 
if $M$ and $N$ are $\lambda$-terms, then so are $(MN)$ and $\lambda x.M$ for each variable $x$.
We denote by $\Lambda^o$ the set of closed $\lambda$-terms. 
Concerning specific $\lambda$-terms we set:
$$
    \begin{array}{c}
    \bold{I}\equiv \lambda x.x, \quad \bold{1} \equiv \lambda xy.xy,\quad \bold{T} \equiv \lambda x y. x, \quad \bold{F} \equiv \lambda x y. y,\\
    \bold{S}\equiv\lambda xyz.xz(yz), \quad\delta \equiv \lambda x.xx, \quad \Omega\equiv \delta\delta, \quad\Omega_3\equiv (\lambda x.xxx)(\lambda x.xxx).\\
    \end{array}
$$
The symbol $\equiv$ denotes definitional equality. 
A more traditional notation for $\bold{T}$, when not viewed as a boolean, is $\bold{K}$.
We will denote $\ga\beta$-conversion by $\lambda_\beta$ and $\ga\beta\eta$-conversion by $\lambda_{\beta\eta}$.

Contexts are, intuitively, $\lambda$-terms with some occurrences of a hole inside, denoted by $[]$.
A \emph{context} is inductively defined as follows: $[]$ is a context, every variable is a context,
if $C_1$ and $C_2$ are contexts then so are $C_1C_2$ and $\lambda x.C_1$ for each variable $x$.
If $M$ is a $\lambda$-term we will write $C[M]$ for the context $C$ where all the occurrences of 
the hole $[]$ have been simultaneously replaced (without $\ga$-conversion) by $M$.

A $\lambda$-term $M$ is a \emph{head normal form} (\emph{hnf}) if 
$M \equiv \lambda x_1\ldots x_n.yM_1\ldots M_k$ for some $n,k\ge 0$.\break
Let $M\equiv \lambda x_1,\ldots,x_n.y M_1\cdots M_k$ and $N\equiv \lambda x_1,\ldots,x_{n'}.y' N_1\cdots N_{k'}$
be two hnf's. 
Then $M,N$ are \emph{equivalent} if, and only if, $y\equiv y'$ and $k-n = k'-n'$.

A $\lambda$-term is \emph{solvable} if it is $\beta$-convertible to a hnf, otherwise it is called \emph{unsolvable}.

\begin{notation}
$\Unsolvable$ denotes the set of all unsolvable $\lambda$-terms.
\end{notation}

\begin{definition} $M,N\in\Lambda^o$ are \emph{separable} if there exists $S\in\Lambda^o$ such that $SM=\bold{T}$ and $SN=\bold{F}$; 
otherwise they are \emph{inseparable}.
\end{definition}
There exist simple criteria implying separability or inseparability.

\begin{proposition}(B\"{o}hm) \cite[Lemma~10.4.1,~Thm.~10.4.2]{Bare}
\begin{itemize}
\item[(i)]
    Two hnf's are separable or equivalent (as hnf's);
\item[(ii)]
    Two normal $\lambda$-terms are separable or $\eta$-equivalent.
\end{itemize}
\end{proposition}

\subsection{B\"ohm trees}\label{bohm}
The \emph{B\"{o}hm tree} $BT(M)$ of a $\lambda$-term $M$ is a finite or infinite labelled tree. 
If $M$ is unsolvable, then $BT(M) = \bot$, that is, $BT(M)$ is a tree with a unique node 
labelled by $\bot$. 
If $M$ is solvable and $\lambda x_{1}\ldots x_{n}.yM_{1}\cdots M_{k}$ is the principal head normal form 
of $M$ \cite[Def.~8.3.20]{Bare} then we have:
        \setlength{\unitlength}{1mm}

        \begin{picture}(30,30)(0,0)

        \put(15,20){\makebox(80,10)[b,l]{$BT(M) = \lambda x_{1} \ldots x_{n}.y$}}
        \put(43,19){\line(-2,-1){12}}
        \put(43,19){\line(2,-1){12}}
        \put(25,7){\makebox(50,10)[b,l]{$BT(M_{1}) \ldots\ldots\ldots BT(M_{k}$)}}
        \end{picture}

We call $\BT$ the set of all B\"ohm trees. 
Given $t,t'\in\BT$ we define $t\BTle t'$ if, and only if, $t$ results from $t'$ by cutting off some subtrees.
It is easy to verify that $(\BT,\BTle)$ is an $\omega$-algebraic cpo.


\subsection{$\lambda$-theories}

A \emph{$\lambda$-theory} is a congruence which contains $\lambda_\beta$.
If $\cT$ is a $\lambda$-theory, we will write $M =_\cT N$ for $(M,N) \in \cT$ and $[M]_\cT$ for the $\cT$-equivalence class of $M$;
for $V\subseteq \Lambda$, $V/\cT$ denotes the quotient set of $V$ modulo $\cT$, i.e., $V/\cT = \{[M]_\cT \st M\in V\}$.
A $\lambda$-theory $\cT$ is: \emph{consistent} if $\cT \neq \Lambda\times\Lambda$, 
\emph{extensional} if it contains  the equation $\bold{I} = \bold{1}$ and \emph{recursively enumerable} if the set of G\"odel 
numbers of all pairs of $\cT$-equivalent $\lambda$-terms is r.e.

The $\lambda$-theory $\cH$, generated by equating all the unsolvable $\lambda$-terms, is consistent 
by \cite[Thm.~16.1.3]{Bare} and admits a unique maximal consistent extension $\cH^*$ \cite[Thm.~16.2.6]{Bare},
which is an extensional $\lambda$-theory. 
A $\lambda$-theory $\cT$ is \emph{sensible} if $\cH\subseteq\cT$. Consistent sensible $\lambda$-theories are never r.e.\ \cite[Thm.~17.1.9]{Bare}.
A $\lambda$-theory $\cT$ is called \emph{semi-sensible} if it contains no equations of the form $U = S$ where $S$ 
is solvable and $U$ unsolvable. 
Sensible $\lambda$-theories are semi-sensible and $\cH^*$ is also the unique maximal semi-sensible $\lambda$-theory.
The $\lambda$-theory $\BTth$ which equates all $\lambda$-terms with the same B\"ohm tree, is sensible, 
non-extensional and non r.e., moreover $\BTth$ is distinct from $\cH$ and $\cH^*$, so that $\cH \subsetneq \BTth\subsetneq\cH^*$.

\section{Models of $\lambda$-calculus}

\subsection{$\lambda$-models}\label{subsec:lambdamodels}
It is well known \cite[Ch.~5]{Bare} that a model of untyped $\lambda$-calculus, or \emph{$\lambda$-model} here, 
is nothing else than a reflexive object of a Cartesian closed category (\emph{ccc} for short)  
$\bold{C}$, that is to say a triple $\lm{M}=(\D,\App,\Abs)$ such that $\D$
is an object of $\bold{C}$ and $(\App,\Abs):[\D\to\D]\to\D$ is a retraction pair, which
means that $\App:\D\to[\D\to\D]$
and $\Abs:[\D\to\D]\to\D$ are morphisms and $\App\comp\Abs=id_{[\D\to \D]}$. In the
following we will only be interested in the case where $\bold{C}$ is a concrete
Cartesian closed category whose objects are posets, possibly
satisfying some constraints, and morphisms are (special) monotone functions
between these sets. In fact we will mainly be interested in Scott semantics but
we will also draw conclusions for the stable and strongly stable semantics. 
In these three classes all the $\lambda$-models have an underlying poset $\D$ which is a cpo;
in this context the partial order $\sqle_{\D}$ will be denoted by $\sqle_\lm{M}$.
An \emph{environment} with values in $\D$ is then a total function $\rho:Var\rightarrow\D,$ where the carrier
set of $\D$ is still denoted by $\D$.

We let $Env_\D$ be the set of environments with values in $\D$.
$Env_\D$, ordered pointwise, is a cpo whose bottom element is the environment $\emptyrho$ mapping everybody to $\bot_\D$.
Note also that $\rho\in Env_{\D}$ is compact if, and only if, 
$\rg(\rho)\subseteq \compel{\D}$ and $\rho(x) \neq \bot_\D$ only for a finite number of $x\in Var$.
For every $x\in Var$ and $d\in \D$ we denote by $\rho[x:=d]$ the environment $\rho'$ which coincides with $\rho$, except on $x$,
where $\rho'$ takes the value $d$.
The interpretation $\Int{M}:Env_\D\to \D$ of a $\lambda$-term $M$ is defined by structural induction on $M$, as follows:

\begin{itemize}
\item
    $\Int{x}_\rho = \rho(x)$,
\item
    $\Int{MN}_\rho = \App(\Int{M}_\rho)(\Int{N}_\rho) $,
\item
    $\Int{\lambda x.M}_\rho = \Abs(d\in \D\mapsto \Int{M}_{\rho[x:=d]})$.
\end{itemize}

This interpretation function generalizes to terms with parameters in $\D$ (where an element of $\D$ is interpreted by itself)
and to $\Lambda_\bot$ by setting $\Int{\bot}_\rho = \bot_\D$ for all $\rho\in Env_\D$.
The set of all open (resp. closed) terms with parameters in $\D$ is denoted by $\Lambda(\D)$ (resp. $\Lambda^o(\D)$).
If $M$ is a closed $\lambda$-term we write $\Int{M}$ instead of $\Int{M}_\rho$ since, clearly,
$\Int{M}_\rho$ only depends on the value of $\rho$ on the free variables of $M$; in particular $\Int{M}=\Int{M}_{\emptyrho}$.
In case of ambiguity we will denote by $\Int{M}^\lm{M}$ the interpretation of the closed term $M$ in the $\lambda$-model $\lm{M}$.

The \emph{equational theory} $\Th{\lm{M}}$ and the \emph{order theory} $\Thle{\lm{M}}$ of a $\lambda$-model $\lm{M}$ are respectively defined as: 
$$
\begin{array}{l}
    \Th{\lm{M}} = \{ (M,N) \st \Int{M}_\gr = \Int{N}_\gr \ \forall \gr \in Env_\D\},\\
    \Thle{\lm{M}} = \{ (M, N) \st \Int{M}_\gr \sqle_\lm{M} \Int{N}_\gr  \ \forall \gr \in Env_\D\}.
\end{array}
$$
The model $\lm{M}$ is called \emph{sensible} (resp. \emph{semi-sensible}) if $\Th{\lm{M}}$ is.

\bigskip

Every $\lambda$-model $\lm{M} = (\D,\App,\Abs)$ can be viewed as the combinatory algebra $\ca{C} = (\D,\bullet,k,s)$ where 
$a \bullet b = \App(a)(b)$ ($a,b\in \D$) and $k,s$ are, respectively, the interpretation of $\bold{K},\bold{S}$ in $\lm{M}$. 
In the sequel, we will write $ab$ for $a \bullet b$, and when parentheses are omitted we understand that association is made to the left, thus $abc$ means $(ab)c$.

\subsection{Isomorphisms of $\lambda$-models}\label{sec:Iso of lambda models}

Given two combinatory algebras $\ca{C}=(\D,\bullet,k,s)$ and $\ca{C}'=(\D',\bullet',k',s')$ a function $\Psi:\D\to\D'$
is a \emph{morphism} from $\ca{C}$ to $\ca{C}'$ if $\Psi (u\bullet v) = \Psi(u)\bullet'\Psi(v)$ and $\Psi (k) = k'$, $\Psi(s)=s'$;
it is an \emph{isomorphism} if and only if $\Psi$ is, moreover, a bijection.

It has been proved in \cite{ManzonettoS06} that all homomorphisms between $\lambda$-models living in the continuous semantics or in its refinements are embeddings, that is to say, inclusions up to isomorphism.

\begin{theorem}\label{2.2} Given two $\lambda$-models $\lm{M} = (\D,\App,\Abs),\lm{M}'= (\D',\App',\Abs')$ and the associated combinatory algebras $\ca{C},\ca{C}'$, 
and a bijection $\Psi:\D\to\D'$, the following assertions are equivalent:
\begin{itemize}
\item[(i)]
    $\Psi$ is an isomorphism between $\ca{C}$ and $\ca{C}'$,
\item[(ii)]
    for all $M\in\Lambda^o$, $\Psi(\Int{M}^\lm{M}) = \Int{M}^{\lm{M}'}$.
\end{itemize}
\end{theorem}

We will hence also speak in this case of an \emph{isomorphism between the $\lambda$-models $\lm{M}$ and $\lm{M}'$},
and of an \emph{automorphism} when $\lm{M}=\lm{M}'$ (and hence $\ca{C} = \ca{C}'$).

The next remark is clear from the definition.

\begin{remark}\label{rem:iso of models} If $\lm{M}$ and $\lm{M}'$ are isomorphic $\lambda$-models, then $\Th{\lm{M}} = \Th{\lm{M}'}$.
\end{remark}

\begin{notation} $Aut(\lm{M})$ is the group of all automorphisms of $\lm{M}$.
\end{notation}

\subsection{Scott continuous semantics}\label{subsec:scottcontfun}
The \emph{Scott-continuous semantics} is the semantics of $\lambda$-calculus given in the category whose objects are cpo's and 
morphisms are Scott-continuous functions.
If $\D$ is a cpo we can define the \emph{Scott topology} on $\D$. 
Given two cpo's $\D,\D'$ a function $f:\D\to \D'$ is Scott-continuous if, and only if, it is monotone and $f(\sup{A}) = \sup{\img{f}(A)}$ 
for all directed $A\subseteq \D$.
We will denote by $[\D\to \D']$ the set of all Scott continuous functions from $\D$ into $\D'$ considered as a cpo by pointwise ordering. 
If $d\in\compel{\D}$ and $e\in\compel{\D'}$ then the step function $\step{d}{e}$, defined as follows, is compact:
$$
    \step{d}{e}(x) = \left\{
    \begin{array}{ll}
	    e & \textrm{if $d\sqle_\D x$,}\\
		\bot_{\D'} & \textrm{otherwise.}\\
	\end{array} 
    \right.
$$
If $\D,\D'$ are Scott domains then $[\D\to \D']$ is a Scott domain and its compact elements are the functions of the form
$\bigsqcup_{i\in I}~\step{d_i}{e_i}$ for some $I$ finite. Note that in case $I\neq\emptyset$ such least upper bound exists if, and only if, whenever
$\{d_i \st i\in I\}$ is bounded, then so is $\{e_i\st i\in I\}$.
For each function $f\in[\D\to\D']$, we define the \emph{trace of $f$} as 
$tr(f) = \{(d,e)\in\compel{\D}\times\compel{\D'} \st e\sqle_{\D'} f(d)\}$.
Note that $(d,e)\in tr(f)$ if, and only if, $\step{d}{e}\sqle_{[\D\to\D']} f$.
If $\D'$ is prime algebraic it is more interesting to work with 
$Tr(f) = \{(d,p)\in\compel{\D}\times\Prime{\D'} \st p\sqle_{\D'}f(d) \}$.
Hence, if $\D,\D' = \pow{D}$ we can use $Tr(f) = \{(a,\alpha)\in D^*\times D \st \alpha\in f(a)\}$.

In the next section we will describe the simplest class of models living in
Scott's continuous semantics, namely graph models. 

\subsection{Definition of graph models}\label{subsec:graphmodels}
The class of graph models belongs to Scott continuous semantics, it is the simplest class of models of the
untyped $\lambda$-calculus; nevertheless it is very rich.
All known classes of webbed $\lambda$-models can be presented as variations of this class (see \cite{Berline00}).
The simplest graph model, is Engeler's model $\gm{E}$ (Example \ref{exa:Engeler-Pomega}$(i)$); it is moreover, from far, 
the simplest of all non syntactical $\lambda$-models.
Historically, the first graph model which has been isolated was Plotkin and Scott's $\gm{P}_\omega$, and it was followed soon by $\gm{E}$.
The word \emph{graph} refers to the fact that the continuous functions are encoded in the model via 
(a sufficient fragment of) their graphs, namely their traces, as recalled below.
For more details we refer to \cite{Berline00}, and to \cite{Berline06}.

\begin{definition}\label{def:totalpair} A \emph{total pair} $\cG$ is a pair $(\setG,\i{\cG})$ where $\setG$ is an infinite set and 
$\i{\cG}:\setG^*\times \setG\to \setG$ is an injective total function.
\end{definition}

\begin{definition}\label{def:graph model} The \emph{graph model} generated by the total pair $\cG$ is the reflexive cpo 
$$\gm{G} = ((\pow{\setG},\subseteq),\Abs^\cG,\App^\cG),$$
where $\Abs^\cG = \img{\i{\cG}}\comp Tr$ and $\App^\cG$ is a left inverse of $\Abs^\cG$. More explicitely:
\begin{itemize}
\item[(i)]
    $\Abs^\cG(f) = \{ \i{\cG}(a,\alpha) \st (a\in \setG^*)\ \alpha \in f(a)\}$,
\item[(ii)]
    $\App^\cG(X)(Y) = \{\alpha\in \setG\st (\exists a\finsubset Y)\ \i{\cG}(a,\alpha)\in X\}$.
\end{itemize}
\end{definition}

In particular, the function $\i{\cG}$ encodes the trace of the Scott continuous function $f:\pow{\setG}\to\pow{\setG}$ by $\Abs^\cG(f)\subseteq \setG$.
The total pair $\cG = (\setG,\i{\cG})$ is called the ``\emph{web}'' of the $\lambda$-model. 

It is easy to check that, in the case of a graph model $\gm{G}$, the interpretation $\Int{M}^\gm{G}:Env_{\pow{\setG}}\to\pow{\setG}$ of $M\in\Lambda$ becomes:
\begin{itemize}
\item
    $\intrho{x}^\gm{G} = \rho(x)$,
\item
    $\intrho{MN}^\gm{G} = \{\alpha\in \setG \st (\exists a\finsubset \intrho{N}^\gm{G})\ \i{\cG}(a,\alpha)\in \intrho{M}^\gm{G}\} $,
\item
    $\intrho{\lambda x.M}^\gm{G} = \{\i{\cG}(a,\alpha) \st (a\in \setG^*)\ \alpha\in \Int{M}^\gm{G}_{\rho[x:=a]}\}$.
\end{itemize}

\begin{example} Given a graph model $\gm{G}$:\\
$\Int{\bold{I}}^\gm{G} \equiv \Int{\lambda x.x}^\gm{G} = \{ \i{\cG}(a,\alpha) \st a\in \setG^* \textrm{ and }\alpha\in a\}$,\\
$\Int{\bold{T}}^\gm{G} \equiv \Int{\lambda xy.x}^\gm{G} = \{ \i{\cG}(a,\i{\cG}(b, \alpha)) \st a,b\in \setG^* \textrm{ and }\alpha\in a\}$,\\
$\Int{\bold{F}}^\gm{G} \equiv \Int{\lambda xy.y}^\gm{G} = \{ \i{\cG}(a,\i{\cG}(b,\alpha)) \st a,b\in \setG^* \textrm{ and }\alpha\in b\}$.
\end{example}

Concerning $\Int{\Omega}^\gm{G}$ we only use the following characterization (the details of the proof are, for example, worked out in \cite[Lemma 4]{BerlineS06}). 

\begin{lemma}\label{omegan} If $\gm{G}$ is a graph model, then $\Int{\Omega}^\gm{G} \equiv \Int{\delta\delta}^\gm{G} = \{\alpha\st (\exists a\subseteq\Int{\delta}^\gm{G})\ \i{\cG}(a,\alpha)\in a\}$. 
\end{lemma}

In the following, ``\emph{graph theory}'' will abbreviate ``the $\lambda$-theory of a graph model''. 

\begin{proposition}\label{prop:no graph theory eq lambdabeta} For all graph model $\gm{G}$, $\Th{\gm{G}}\neq \lambda_{\beta},\lambda_{\beta\eta}$.
\end{proposition}

Indeed, it was long ago noticed that no graph model could be extensional, and recently noticed in 
\cite{BucciarelliS04} that $\Int{\Omega_{3}}^{\gm{G}}\subseteq\Int{\bold{1}\Omega_{3}}^{\gm{G}}$ holds in all graph models $\gm{G}$
(because $\Int{\Omega_{3}}^{\gm{G}}\subseteq \rg(\i{\cG})$). 
Hence, Selinger's result \cite[Cor.~4]{Selinger96} stating that in any partially ordered model whose theory is $\lambda_\beta$ or $\lambda_{\beta\eta}$ the interpretations 
of closed $\lambda$-terms are discretely ordered, implies that the theory of a graph model cannot be $\lambda_{\beta},\lambda_{\beta\eta}$.

\subsection{The stable and strongly stable semantics}

The stable semantics and the strongly stable semantics are refinements of
Scott's semantics which were successively introduced respectively by Berry \cite{Berry78,BerryTh} and Ehrhard \cite{BucciarelliE91},
mainly for proving some properties of typed $\lambda$-calculi with a flavour of sequentiality \cite{Plotkin77,Berry78,BerryTh,BerryC85}. 
For this paper it is enough to know the following. 
In this framework, the objects are particular prime algebraic Scott domains called \emph{DI-domains} (resp. \emph{DI-domains with coherences}) where, in particular, $u\sqcap v$ is defined
for all pairs $(u,v)$ of compatible elements. 
The morphisms are, respectively, the stable and strongly stable functions between such domains.

A function between $DI$-domains is \emph{stable} if it is Scott continuous and furthermore
commutes with ``inf's of compatible elements''. 
A \emph{strongly stable function} between DI-domains with coherence, is a stable function, preserving coherence. 
The relevant order on the corresponding cpo's of functions, respectively $[\D\to_s\D]$ and $[\D\to_{ss}\D]$ is, in both cases,
Berry's order $\leq_{s}$ which is defined as follows.

\begin{notation}
$f\leq_{s} g$ if, and only if, $\forall x\forall y\;(x\sqle_\D y\imp f(x)=f(y)\sqcap g(x))$
\end{notation}

The following basic properties of Berry's order are easy to check.

\begin{remark}\label{rem:Berrysproperties}\ 
\begin{itemize}
\item[(i)] $f\leq_{s}g$ implies that $f$ is pointwise smaller than $g$,
\item[(ii)] $f\leq_{s}g$ and $g$ constant imply $f$ constant.
\end{itemize}
\end{remark}

As soon as we are working with stable functions, the following alternative notion 
of trace makes sense and it is more economical: $Tr_{s}(f)$ is defined in the same way as $Tr(f)$ in Section \ref{subsec:scottcontfun} (case where $\D$ is prime algebraic)
but retains only the pairs $(d,e)$ satisfying: $d$ is \emph{minimal} such that $e\sqle_\D f(d);$ 
and similarly when one uses pairs $(a,\alpha)$. 
For example, if $\D=(\pow{D},\subseteq)$, for some set $D$, then 
$Tr(id_\D)=\{(a,\alpha)\st\alpha\in a\in D^*\}$ while 
$Tr_{s}(id_{\D})= \{(\{\alpha\},\alpha)\st\alpha\in D\}$.

\section{Classes of webbed models}

We have the following classes of webbed models: 
\begin{enumerate}
\item $K$-models
introduced by Krivine in \cite{Krivine93} (see also \cite[Def.126]{Berline00}), pcs-models 
\cite[Def.153]{Berline00}, and filter models \cite{CoppoDZ87}, all living in the continuous semantics;
\item Girard's reflexive coherences, called $G$-models in \cite[Def.150]{Berline00}, living in the stable semantics;
\item Ehrhard's reflexive hypercoherences, called $H$-models in \cite[Def.160]{Berline00}, living in the strongly stable semantics.
\end{enumerate}
The terminology
of $K$-, $G$-, $H$- models will be used freely in this paper.

\section{Recursion in $\lambda$-calculus}\label{sec:Recursion in lambda-calculus}

We now recall the main properties of recursion theory concerning $\lambda$-calculus that will be applied in the following sections.

Let $(-)_\omega :\Lambda \to \nat$ be an arbitrary effective encoding of $\Lambda$. 
We denote by $(-)_\lambda$ the inverse map of $(-)_\omega$, thus: $M_{\omega,\lambda} = M$. 

\subsection{Co-r.e.\ sets of $\lambda$-terms }\label{sec:co-r.e. sets}

\begin{definition}
A set $V\subseteq \Lambda$ is \emph{r.e.} (\emph{co-r.e.}) if $\{(M)_\omega\st M\in V\}$ is r.e.\ (co-r.e.).
The set $V$ is called \emph{trivial} if either $V = \emptyset$ or $V = \Lambda$.
\end{definition}

\begin{notation}
Let $\cT$ be a $\lambda$-theory. An r.e.\ (co-r.e.) set  of $\lambda$-terms closed under $=_\cT$ will be called a {\em $\cT$-r.e.} ({\em $\cT$-co-r.e.}) set.
If $\cT = \lambda_{\beta}$ we simply speak of a $\beta$-r.e.\ ($\beta$-co-r.e.) set.
\end{notation}

\begin{definition} A family $(X_i)_{i\in I}$ of sets has the {\em FIP} (finite intersection property) if 
$X_{i_1}\cap\dots\cap X_{i_n}\neq \emptyset$ for all $i_1,\dots,i_n\in I$.
\end{definition}

Our key tool for studying r.e.\ theories and effective models will be the following theorem (\cite[Thm.~2.5]{Visser80}, or \cite[Ch. 17]{Bare}).

\begin{theorem}\label{thm:hyperconnection} (Visser) The family of all non-empty $\beta$-co-r.e.\ subsets of $\gL$ has the FIP.
\end{theorem}

This theorem generalizes the following classical result of Scott (see, e.g., \cite[Thm.~6.6.2]{Bare}).

\begin{theorem} (Scott) A set of $\lambda$-terms which is both $\beta$-r.e.\ and $\beta$-co-r.e.\ is trivial.
\end{theorem}

A topological reading of Theorem~\ref{thm:hyperconnection} is that the topology on $\gL$ generated by the $\beta$-co-r.e.\ sets of $\lambda$-terms 
is hyperconnected (i.e., the intersection of two non-empty open sets is non-empty).

\begin{lemma} \ 
\begin{itemize}
\item[(i)] $\Unsolvable$ is $\beta$-co-r.e.\ and hence non r.e.,
\item[(ii)] $\Unsolvable$ is $\cT$-co-r.e.\ if, and only if, $\cT$ is semi-sensible.
\end{itemize} 
\end{lemma}

\begin{proof} 
$(i)$ Indeed, $\Unsolvable$ is co-r.e and $\beta$-closed. \\
$(ii)$ Furthermore, it is easy to check that $\Unsolvable$ is $\cT$-closed exactly when $\cT$ is semi-sensible. 
\end{proof}

From this lemma and from Theorem~\ref{thm:hyperconnection} it follows that every non-empty $\beta$-co-r.e.\ set of terms contains unsolvable $\lambda$-terms.

\begin{lemma}\label{lemma:OinfmoduloT}
If $O\neq\emptyset$ is $\beta$-co-r.e.\ and $\cT$ is r.e., then $O/\cT$ is infinite or $\cT$ is inconsistent.
\end{lemma}

\begin{proof} 
Let $V$ be the $\cT$-closure of $O$, and $O'= \Lambda\setminus V$. 
If $\cT$ is r.e.\ and $O/\cT$ is finite, then $V$ is r.e.\ and hence $O'$ is $\beta$-co-r.e.\ 
Since $O'\cap O=\emptyset$, $O'$ must be empty by Theorem~\ref{thm:hyperconnection}.
Hence $V=\Lambda$, and $\Lambda/\cT$ is finite. Hence $\cT$ is inconsistent. 
\end{proof}

\subsection{Separability revisited}

This section, which can be skipped at first reading, contains other interesting examples of $\beta$-co-r.e.\ sets which should prove
useful for later work, namely $\Tins{M}$ and $\Lambda_{\cT-easy}$, as defined below, when $\cT$ is r.e. 

\begin{notation} Given a $\lambda$-theory $\cT$ we let, for all $M\in\Lambda^o$:
$$  
    \Tins{M}=\{ N\in\Lambda^o \st \nexists S\in\Lambda^o (SM =_\cT \bold{T}\ \land SN=_\cT \bold{F}) \}.
$$
In this case $M$ and $N$ are said to be \emph{$\cT$-inseparable}. We will omit $\cT$ when $\cT = \lambda_\beta$. 
\end{notation}



\begin{example}\ 
\begin{itemize}
\item[(i)] 
    $\Unsolvable\subseteq M^{ins}$ for all $M\in\Lambda^o$ (by the genericity lemma \cite[Prop.~14.3.24]{Bare}); equivalently:
\item[(ii)] 
    $U^{ins}=\Lambda^o$ for all unsolvable terms $U$; more generally:
\item[(iii)] 
    If $BT(M)\BTle BT(N)$, then $M\in N^{ins}$ and $N\in M^{ins}$.
\end{itemize}
\end{example}

\begin{remark} Let $\cT$ be a $\lambda$-theory and $M\in\Lambda^o$, then:
\item[(i)] 
    $\Tins{M}\subseteq M^{ins}$ is $\cT$-closed; more generally:
\item[(ii)] 
    $M^{\cT'-ins}\subseteq \Tins{M}$ if $\cT\subseteq\cT'$.
\end{remark}

\begin{proposition} Suppose $\cT$ is a semi-sensible $\lambda$-theory and $M\in\Lambda^o$, then $\Tins{M} = M^{ins}$.
\end{proposition} 

\begin{proof} Suppose that
there are $S,N\in\Lambda^o$ such that 
$SM=_{\cT}\bold{T}$ and $SN=_{\cT}\bold{F}$. As $\cT$ is semi-sensible, $SM$ and
$SN$ are solvable, and, since $\cT$ is necessarily consistent, their head-normal forms are non equivalent. 
Hence $SM$ and $SN$ are separable, which implies that $M$ and $N$ are separable.
\end{proof}

\begin{definition}\label{def:easy-term}
Let $\cT$ be a $\lambda$-theory. 
A $\lambda$-term $U\in\Lambda^o$ is called \emph{$\cT$-easy} if, for all $M\in \Lambda^o$, $U=M$ is consistent with $\cT$ \cite[p.~434]{Bare}. 
\end{definition}

\begin{notation}
We will denote by $\Lambda_{\cT-easy}$ the set of $\cT$-easy terms. 
\end{notation}
It is clear that $\Lambda_{\cT-easy}\subseteq \Unsolvable$ and that $\Lambda_{\cT-easy}\subseteq\Tins{M}$ for all $M\in\Lambda^o$. 

\begin{proposition} If $\cT$ is a consistent r.e.\ $\lambda$-theory then:
\begin{itemize}
\item[(i)] 
    $\Tins{M}$ and $\Lambda_{\cT-easy}$ are $\cT$-co-r.e. sets,
\item[(ii)]
    $\Tins{M}$ and $\Lambda_{\cT-easy}$ are infinite modulo $\cT$.
\end{itemize}
\end{proposition}

\begin{proof} Recall that $\Lambda_{\cT-easy}\subseteq M^{\cT-ins}$ and $\Lambda_{\cT-easy}\subseteq \Unsolvable$ hold whether $\cT$ is r.e. or not.\\
$(i)$ It follows easily from the definitions that both sets are co-r.e.\ if $\cT$ is r.e.\\
$(ii)$ By $(i)$ and Lemma \ref{lemma:OinfmoduloT}, using the fact that $\Lambda_{\cT-easy}\neq\emptyset$ when $\cT$ is r.e.\ was proved by Visser \cite{Visser80} (or see \cite[Prop.~17.1.9]{Bare}).
\end{proof}

Note that, if it is obvious that $\Tins{M}\neq\emptyset$ holds for all $M\in\Lambda^o$ and $\cT$, since $M\in\Tins{M}$,
it is only known for r.e.\ theories $\cT$ that $\Lambda_{\cT-easy}$ is non-empty.

\part{Graph models and partial pairs}

\section{The category of partial pairs}\label{sec:The category of partial pairs}

The definition of graph models (and hence of total pairs) has been recalled in Section \ref{subsec:graphmodels}.
We need now to develop the wider framework of partial pairs.

In this section will recall the known definitions of partial pairs, interpretation with respect to a partial pair, free completion and gluings.
We will also introduce the new notions of subpair relation, morphism of partial pairs
and retract of partial pairs.

\subsection{Definition and ordering of partial pairs}

\begin{definition}\label{def:partialpair} A \emph{partial pair} $\cA$ is a pair $(A,j_\cA)$ where $A$ is a non-empty set and $j_\cA:A^*\times A\to A$ is a partial
(possibly total) injection.
\end{definition}

In the sequel the letters $\cA,\cB$ will always denote partial pairs.
$\cA$ is \emph{finite} if $A$ is finite, and it is total if $j_\cA$ is total.
The simplest example of a partial pair is $(A,\emptyset)$, where $\emptyset$ denotes the empty function. 

\begin{definition}\label{def:subpair}
$\cA$ is a \emph{subpair} of $\cB$, written $\cA\subpair \cB$,
if $A\subseteq B$ and  $j_\cA(a,\alpha) = j_\cB(a,\alpha)$ for all $(a,\alpha)\in \dom(j_\cA)$.
The set of all the subpairs of $\cA$ will be denoted by $Sub(\cA)$.
\end{definition}

It is clear that, for all partial pairs $\cA$, 
$(Sub(\cA),\subpair)$ is a bounded complete algebraic cpo (provided we add the empty-pair) and that
$\sup_{k\in K} \cA_k = (\cup_{k\in K} A_k, \cup_{k\in K} j_{\cA_k})$, if the $\cA_k$'s are compatible.
When $A$ is countable, $(Sub(\cA),\subpair)$ is even a $DI$-domain.




\subsection{Interpretation with respect to partial pairs}

\begin{definition}\label{def:A-env}
An \emph{$\cA$-environment} is a function $\rho:Var\to \pow{A}$.
\end{definition}

We will denote by $Env_\cA$, instead of $Env_{\pow{A}}$, the set of all $\cA$-environments.

The definition of the \emph{interpretation} $\Int{M}^\cA$ of a $\lambda$-term $M$ with respect to a partial pair $\cA$ generalizes in the obvious way the one 
given for graph models in Section \ref{subsec:graphmodels}. For all $\rho\in Env_\cA$ we let:
\begin{equation}
\Int{x}_{\rho }^{\cA}=\gr(x),
\end{equation}
\begin{equation}
\Int{PQ}_{\rho }^{\cA}=\{\alpha\st(\exists a\subseteq \Int{Q}_{\rho}^{\cA})\ (a,\alpha )\in \dom(j_\cA)\land j_{\cA}(a,\alpha)\in \Int{P}_{\rho }^{\cA}\},
\end{equation}
\begin{equation}
\Int{\lambda x.N}_{\rho }^{\cA}=\{j_{\cA}(a,\alpha)\st(a,\alpha )\in \dom(j_\cA)\wedge \alpha \in \Int{N}_{\rho [x:=a]}^{\cA}\,\}.
\end{equation}
Of course, if $\gm{G}$ is a graph model with web $\cG$, then $\Int{M}^\gm{G}_\rho = \Int{M}^\cG_\rho$ for all $\lambda$-terms $M$ and environments $\rho$.
Note that, if $\cA$ is not total, $\beta$-equivalent $\lambda$-terms do not necessarily have the same interpretation.

\begin{notation} If $\rho\in Env_\cA, \sigma\in Env_\cB$ and $C$ is a set, then $\sigma = \rho\cap C$ means
$\sigma(x) = \rho(x)\cap C$ for every variable $x$, and $\rho\subseteq \sigma$ means $\rho(x)\subseteq \sigma(x)$ for every variable $x$.
\end{notation}

We now provide two new lemmata which express the continuity of the function defined from $Sub(\cA)\times Env_\cA$ to $\pow{A}$ and mapping 
$(\cB,\rho)\mapsto \Int{M}^\cB_{\rho\cap B}$.

\begin{lemma}\label{lemma:onenvironments} 
If $\cA\subpair\cB$, then $\Int{M}^\cA_\rho\subseteq \Int{M}^\cB_\sigma$ for all $\rho\in Env_\cA$ and $\sigma\in Env_\cB$ such that $\rho\subseteq\sigma$.
\end{lemma}

\begin{proof} By straightforward induction on the structure of $M$.
\end{proof}

\begin{lemma}\label{lemma:existenceoffinitesubpairs} Let $M\in\Lambda$, $\cA$ be a partial pair and $\rho\in Env_\cA$.
Suppose $\alpha\in \Int{M}^{\cA}_\rho$ then there exists a finite pair $\cB\subpair\cA$ such that $\alpha\in \Int{M}^\cB_{\rho\cap B}$.
\end{lemma}

\begin{proof} The proof is by induction on $M$.

If $M\equiv x$, then $\alpha\in \rho(x)$, so that we define $\cB = (\{\alpha\},\emptyset)$.

If $M\equiv PQ$, then there is $a = \{ \alpha_1,\dots,\alpha_n\}$, for some $n\ge 0$, such that 
$(a,\alpha)\in \dom(j_\cA)$, $j_\cA(a,\alpha)\in \Int{P}^{\cA}_\rho$ and $a\subseteq \Int{Q}^{\cA}_\rho$.
By induction hypothesis there exist finite subpairs $\cB_1,\dots,\cB_{n+1}$ of $\cA$ 
such that $j_\cA(a,\alpha)\in \Int{P}^{\cB_{n+1}}_{\rho \cap B_{n+1}}$
and $\alpha_k\in \Int{Q}^{\cB_k}_{\rho \cap B_k}$ for $k=1,\dots,n$.  
We define $\cB\subpair\cA$ as $\sqcup_{k=0...n+1}\cB_k$ where $\cB_0 = (a\cup\{\alpha\}, j_\cA\restr_{\{(a,\alpha)\}})$.
From Lemma~\ref{lemma:onenvironments} it follows the conclusion. 

If $M\equiv \lambda x.N$, then $\alpha = j_\cA(b,\beta)$ for some $b$ and $\beta$ such that $(b,\beta)\in \dom(j_\cA)$ and $\beta\in \Int{N}^{\cA}_{\rho[x:=b]}$.
By induction hypothesis there exists a finite pair $\cC\subpair\cA$ such that $\beta\in \Int{N}^\cC_{\rho[x:=b]\cap C}$.
We define $\cB\subpair\cA$ as $\cC\sqcup (b\cup\{\alpha,\beta\},j_\cA\restr_{\{(b,\beta)\}})$.
Then we have that $\cC\subpair\cB$ and $\rho[x:=b]\cap C\subseteq \rho[x:=b]\cap B$.
From $\beta\in \Int{N}^\cC_{\rho[x:=b]\cap C}$ and from Lemma~\ref{lemma:onenvironments} it follows that 
 $\beta\in \Int{N}^{\cB}_{\rho[x:=b]\cap B}= \Int{N}^{\cB}_{(\rho\cap B)[x:=b]}$.
Then we conclude that $\alpha = j_\cB(b,\beta) \in \Int{\lambda x.N}^{\cB}_{\rho\cap B}$.
\end{proof}

\subsection{Morphisms between partial pairs}


The following definition extends the definition of an isomorphism between total pairs, which was introduced by Longo in \cite{Longo83}. 

\begin{definition} 
A total function $\theta:A\to B$ is a \emph{morphism} from $\cA$ to $\cB$ if, for all $(a,\alpha)\in A^*\times A$, we have:
$$
    (a,\alpha)\in \dom(j_\cA)\Longrightarrow[(\img{\theta}(a),\theta(\alpha))\in \dom(j_\cB)\textrm{ and }\theta(j_\cA(a,\alpha)) = j_\cB(\img{\theta}(a),\theta(\alpha))]
$$ 
and it is an \emph{endomorphism} if, moreover, $\cA = \cB$.
\end{definition}

\begin{remark}\label{rem:incusionmapp} \ 
\begin{itemize}
\item[(i)]
    $\theta:A\to B$ is an \emph{isomorphism} between $\cA$ and $\cB$ if, and only if, it is a bijection and both $\theta$ and $\partinv{\theta}$ are morphisms;
    if, moreover, $\cA = \cB$ then $\theta$ is an \emph{automorphism}.
\item[(ii)]
    $\cA\subpair\cB$ if, and only if, the inclusion mapping $\iota:A\to B$ is a morphism.
\end{itemize}
\end{remark}

\begin{notation} \ 
\begin{itemize}
\item[(i)]
    $Hom(\cA,\cB)$ denotes the set of morphisms from $\cA$ to $\cB$,
\item[(ii)]
    $Iso(\cA,\cB)$ denotes the set of isomorphisms between $\cA$ and $\cB$,
\item[(iii)]
    $Aut(\cA)$ denotes the group of automorphisms of $\cA$. 
\end{itemize}
\end{notation}

We will also write $\theta:\cA\to\cB$ for $\theta\in Hom(\cA,\cB)$.

\begin{lemma}\label{lemma:on iso and plus} Let $\phi\in Hom(\cA,\cB)$ and $\rho\in Env_\cA$. Then:
\begin{itemize}
\item[(i)]
    $\img{\phi}(\Int{M}^\cA_\rho)\subseteq \Int{M}^\cB_{\img{\phi}\comp \rho}$;
\item[(ii)]
    $\img{\phi}(\Int{M}^\cA_\rho) = \Int{M}^\cB_{\img{\phi}\comp \rho}$ if $\phi\in Iso(\cA,\cB)$.
\end{itemize}
\end{lemma}

\begin{proof} $(i)$ By straightforward induction on $M$ one proves that, for all $\alpha\in \img{\phi}(\Int{M}^\cA_\rho)$, we have $\phi(\alpha)\in\Int{M}^\cB_{\img{\phi}\comp \rho}$.\\
$(ii)$ By $(i)$ it is enough to prove that $\Int{M}^\cB_{\img{\phi}\comp \rho}\subseteq\img{\phi}(\Int{M}^\cA_\rho)$.
Let $\psi = \partinv{\phi}$; then $\img{\phi}\comp \img{\psi} = id$. 
Thus $\Int{M}^\cB_{\img{\phi}\comp\rho} = \img{\phi}(\img{\psi}(\Int{M}^\cB_{\img{\phi}\comp\rho}))\subseteq 
\img{\phi}(\Int{M}^\cA_{\img{\psi}\comp\img{\phi}\comp\rho}) = \img{\phi}(\Int{M}^\cA_{\rho})$ (the inclusion follows by $(i)$).
\end{proof}

Lemma \ref{lemma:on iso and plus}$(ii)$ implies that if $\phi\in Iso(\cA,\cB)$, with $\cA,\cB$ total, then $\img{\phi}$ is an isomorphism of $\lambda$-models
(and of combinatory algebras).
On the contrary, if $\phi$ is only a morphism of pairs, then $\img{\phi}$ cannot be a morphism of combinatory algebras.
Indeed, it is easy to check that $\img{\phi}(\Int{\bold{K}}^\cA)\subsetneq \Int{\bold{K}}^\cB$ if $\phi$ is not surjective and 
$\img{\phi}(\Int{MN}^\cA)\subsetneq \img{\phi}(\Int{M}^\cA)\bullet\img{\phi}(\Int{N}^\cA)$ if $\phi$ is not injective.

\subsection{Free-completions of partial pairs}

There are two known processes for building a graph model satisfying some additional requirements. 
Both consist in completing a partial pair $\cA$ into a total pair. 
The \emph{free completion}\footnote{
Free completion is termed \emph{canonical completion} in \cite{BucciarelliS08} and \emph{Engeler completion} in \cite{BucciarelliS03,BucciarelliS04}.
}, which 
is due to Longo \cite{Longo83} and mimics the construction of $\gm{E}$, is a constructive way for building as freely as possible 
a total pair $\ppcompl{\cA}$ from a partial pair $\cA$. 
The aim is to induce some properties of the graph model generated by $\ppcompl{\cA}$ from properties of $\cA$. 
The other completion process, called \emph{forcing completion} or simply ``forcing'', originates in \cite{BaetenB79}.
For all $M\in\Lambda^o$, Baeten and Boerboom built out of a partial pair $(\setG,\emptyset)$ a graph model $\gm{G}$ with web $(\setG,i^M_{\cG})$ such that 
$\Int{\Omega}^{\gm{G}} = \Int{M}^{\gm{G}}$, thus proving semantically that $\Omega$ is easy.
This technique is, in general, non constructive but it can be effective in some degenerate but interesting cases (this contradicts a remark in \cite[Sec.~5.3.5]{Berline06}).
Forcing was generalized in \cite{BerlineS06}, where it is shown, in particular, that we can 
go far beyond $\Lambda^o$ and even $\Lambda^o(\D)$. 
In our paper forcing will only have an auxiliary role allowing us to produce examples; hence ``completion'' will mean ``free completion''
unless otherwise stated.

\begin{definition}\label{def:completionpartialpair} (Longo)
Let $\cA = (A,j_\cA)$ be a partial pair. 
The \emph{free completion}\footnote{
    Actually this completion construction requires that $((A^*\times A)\setminus \dom(j_\cA))\cap \rg(j_\cA) = \emptyset$, otherwise $\funcompl{\cA}$
    would not be injective, hence we will always suppose that no element of $A$ is a pair. 
    This is not restrictive because partial pairs can be considered up to isomorphism. 
} of $\cA$ is the total pair $\ppcompl{\cA}=(\setcompl{A},\funcompl{\cA})$, where $\setcompl{A} = \cup_{n\in\nat}A_n$, with 
$A_0 = A,\ A_{n+1} = A \cup ((A_n^*\times A_n)\setminus \dom(j_\cA))$ and $\funcompl{\cA}$ is defined by:
$$
\funcompl{\cA}(a,\alpha)=\left\{\begin{array}{ll}
		j_\cA(a,\alpha) & \textrm{ if } (a,\alpha)\in \dom(j_\cA),\\
		(a,\alpha) & \textrm{otherwise.}\\
		\end{array} \right.
$$
An element of $A$ has rank $0$, whilst an element $\alpha\in \setcompl{A}\setminus A$ has rank $n$ if $\alpha\in A_n\setminus A_{n-1}$.
\end{definition}


\begin{notation} 
$\gmcompl{\cA}$ denotes the graph model whose web is $\ppcompl{\cA}$, and it will be said \emph{freely generated} by $\cA$. 
\end{notation}

\begin{theorem}\label{thm:no semi-sensible free completions} (Bucciarelli and Salibra~\cite[Thm.~29]{BucciarelliS08})\\
If $\cA$ is a partial pair which is not total then $\gmcompl{\cA}$ is semi-sensible.
\end{theorem}

\begin{remark}\label{rem:ShouldBeTrue} Let $\cA,\cB$ be two partial pairs. If $\cA\subpair\cB\subpair \ppcompl{\cA}$ then $\ppcompl{\cA} = \ppcompl{\cB}$ and hence 
$\gmcompl{\cA} = \gmcompl{\cB}$.
\end{remark}

\begin{example}\label{exa:Engeler-Pomega} By definition:  
\begin{itemize} 
\item[(i)] the Engeler model $\gm{E}$ is 
           freely generated by $\cA=(A,\emptyset)$, where $A$ is a non-empty set.
Thus, in fact, we have a family of graph models $\gm{E}_A$;
\item[(ii)]
    the graph-Scott models are freely generated by $\cA=(A,j_\cA)$, where $j_\cA(\emptyset,\alpha)=\alpha$ for all $\alpha\in A$;
\item[(iii)]
    the graph-Park models are freely generated by $\cA=(A,j_\cA)$, where $j_\cA(\{\alpha\},\alpha)=\alpha$ for all $\alpha\in A$;
\item[(iv)]
    the mixed-Scott-Park graph models are freely generated by $\cA=(A,j_\cA)$ where 
    $j_\cA(\emptyset,\alpha)=\alpha$ for all $\alpha\in Q$, $j_\cA(\{\beta\},\beta)=\beta$ for all $\beta\in R$
    and $Q,R$ form a non-trivial partition of $A$. 
\end{itemize}
\end{example}

\begin{remark}\label{rem:Pomega} (Longo \cite{Longo83}) 
The model $\gm{P}_{\omega}$ is isomorphic to the graph-Scott model associated with any singleton set $A$.
\end{remark}

\begin{theorem} (Kerth \cite{Kerth94,Kerth98b}) There exist $2^{\aleph_0}$ graph models of the form $\gmcompl{\cA}$, with distinct theories, 
among which $\aleph_0$ are freely generated by finite pairs. 
The same is true for sensible graph models.
\end{theorem}

\begin{proof}
Among the continuum of distinct graph models provided by Kerth in \cite{Kerth94,Kerth98b}, countably many are freely generated by finite 
partial pairs. 
The result for sensible graph theories follows from Kerth \cite{Kerth98} plus David \cite{David01}. 
\end{proof}

\begin{lemma}\label{lemma:homo-isomorphism} \ 
\begin{itemize}
\item[(i)]
    For all $\theta\in Hom(\cA,\cB)$ there is a unique $\bar{\theta}\in Hom(\setcompl{\cA},\setcompl{\cB})$
such that $\bar{\theta}\restr_A = \theta$.
\item[(ii)]
    If $\theta\in Iso(\cA,\cB)$, then $\bar{\theta}\in Iso(\setcompl{\cA},\setcompl{\cB})$.
\end{itemize}
\end{lemma}

\begin{proof} Definition of $\bar{\theta}$ and verification of the first point are by straightforward induction on the rank of the elements of 
$\setcompl{A}$. It is also easy to check that if $\theta$ is an isomorphism then $\partinv{\bar{\theta}}$ is the inverse of $\bar{\theta}$.
\end{proof}

A morphism $\theta:\cA\to\cB$ does not induce, in general, a morphism of $\lambda$-models. But this is true when $\theta$ is an isomorphism.
In other words, the next corollary holds.

\begin{corollary} Let $\theta\in Iso(\cA,\cB)$, then:
\begin{itemize}
\item[(i)]
    $\img{\bar{\theta}}\in Iso(\gmcompl{\cA},\gmcompl{\cB})$,
\item[(ii)]
    $\Thle{\gmcompl{\cA}} = \Thle{\gmcompl{\cB}}$,
\item[(iii)]
    $\Th{\gmcompl{\cA}} = \Th{\gmcompl{\cB}}$.
\end{itemize}
\end{corollary}

\begin{proof} $(i)$ By Lemma~\ref{lemma:on iso and plus} and Lemma~\ref{lemma:homo-isomorphism}. 
$(ii)$ By $(i)$ and Remark~\ref{rem:iso of models}.
$(iii)$ From $(ii)$.
\end{proof}

\begin{proposition}\label{claimclaim} 
Let $\gm{G}$ be a graph model with web $\cG$, and suppose $\alpha\in \Int{M}^\gm{G} \setminus \Int{N}^\gm{G}$ for some $M,N\in\Lambda^o$.
Then there exists a finite  $\cA\subpair\cG$ such that $\alpha\in\cA$ and for all pairs $\cC\suppair \cA$, if there is a morphism
$\theta:\cC\to\cG$ such that $\theta(\alpha) = \alpha$, then $\alpha\in \Int{M}^\cC \setminus \Int{N}^\cC$.
\end{proposition}

\begin{proof} 
By Lemma~\ref{lemma:existenceoffinitesubpairs} there is a finite $\cA\subpair\cG$ such that $\alpha\in \Int{M}^\cA$. 
By Lemma~\ref{lemma:onenvironments} we have $\alpha\in \Int{M}^\cC$. 
Now, if $\alpha\in \Int{N}^\cC$ then, by Lemma~\ref{lemma:on iso and plus}, $\alpha = \theta(\alpha)\in \Int{N}^\gm{G}$, which is a contradiction.
\end{proof}

\begin{corollary}\label{cor:alphainBandcompletion} Let $\gm{G}$ be a graph model, and suppose $\alpha\in \Int{M}^\gm{G} \setminus \Int{N}^\gm{G}$ for some $M,N\in\Lambda^o$.
Then there exists a finite  $\cA\subpair\cG$ such that $\alpha\in\cA$ and for all pairs $\cB$ satisfying $\cA\subpair\cB\subpair\cG$, we have: 
\begin{itemize}
\item[(i)]
    $\alpha\in \Int{M}^\cB \setminus \Int{N}^\cB$ and 
\item[(ii)]
    $\alpha \in \Int{M}^{\gmcompl{\cB}}\setminus\Int{N}^{\gmcompl{\cB}}$.
\end{itemize}
\end{corollary}

\begin{proof} 
We apply Proposition~\ref{claimclaim}, taking for $\theta$ the inclusion mapping $\iota:\cB\to\cG$ for $(i)$, 
and $\bar{\iota}$ given by Lemma~\ref{lemma:homo-isomorphism} for $(ii)$.
\end{proof}

\subsection{Retracts}

\begin{definition}
Given two partial pairs $\cA$ and $\cB$ we say that \emph{$\cA$ is a retract of $\cB$}, and we write $\cA\retract \cB$, if there are morphims 
$e\in Hom(\cA,\cB)$ and $\pi\in Hom(\cB,\cA)$ such that $\pi\comp e =id_{A}$. 
In this case we will also write $e,\pi:\cA\retract \cB$. 
\end{definition}

\begin{notation} Given two graph models $\gm{G}, \gm{G'}$ we write $\gm{G}\retract \gm{G'}$ if $\cG\retract\cG'$.
\end{notation}

From Lemma~\ref{lemma:homo-isomorphism}$(i)$, and the fact that $id_{\ppcompl{\cA}}$ is the only endomorphism of $\ppcompl{\cA}$ whose restriction to $A$ is the identity 
$id_{A}$ we get the following lemma.

\begin{lemma}\label{Lemma Extending retractions}
Let $\cA,\cB$ be two partial pairs, then $\cA\retract \cB$ implies $\ppcompl{\cA}\retract \ppcompl{\cB}$.
\end{lemma}

\begin{proposition}\label{Proposition Retract}
If $\cG\retract\cG'$ then $\Thle{\gm{G'}}\subseteq \Thle{\gm{G}}$ and $\Th{\gm{G'}}\subseteq \Th{\gm{G}}.$
\end{proposition}

\begin{proof}
Let $\pi, e:\cG\retract \cG'$. 
It is enough to prove that for all $M,N\in \Lambda^o$, if $\alpha \in \Int{M}^\gm{G}\setminus \Int{N}^\gm{G}$ then 
$e (\alpha )\in \Int{M}^{\gm{G'}}\setminus\Int{N}^{\gm{G'}}$. 
Now, by applying Lemma~\ref{lemma:on iso and plus} twice, $e (\alpha )\in \Int{M}^{\gm{G'}}$ and $e (\alpha )\in \Int{N}^{\gm{G'}}$ would imply $\alpha
=\pi(e (\alpha ))\in \Int{N}^\gm{G}$.
\end{proof}

\begin{example}\label{Example Engeler Plotkin} 
For the Engeler model $\gm{E}_A = \gmcompl{\cA}$ where $\cA = (A,\emptyset)$ and the graph-Scott model $\gm{P}_A = \gmcompl{\cA'}$ where $\cA' = (A,j_\cA)$, and for all $\alpha\in A$ we have:
\begin{itemize}
\item[(a)] $\cA\subpair \cA'$ but \emph{not} $\cA\subpair \cA'\subpair\ppcompl{\cA}$,

\item[(b)] $(\emptyset,\alpha)\in \setcompl{A}\setminus \setcompl{A'}$, 

\item[(c)] $\Th{\gm{E}_A}=\Th{\gm{P}_A} = \BTth$ \cite{Longo83},

\item[(d)] $\Thle{\gm{E}_A}\subsetneq \Thle{\gm{P}_A}$ \cite[Prop.~2.8]{Longo83},

\item[(e)] $\bold{I}\sqle \varepsilon \in \Thle{\gm{P}_A}\setminus \Thle{\gm{E}_A}$ (easy),

\item[(f)] $\alpha\in \Int{\lambda x.\bold{I}}^{\gm{P}_A}\setminus \Int{\lambda x.\bold{I}}^{\gm{E}_A}$ while 
$(\emptyset,\alpha)\in \Int{\lambda x.\bold{I}}^{\gm{E}_A}\setminus\Int{\lambda x.\bold{I}}^{\gm{P}_A}$.
\end{itemize}
\end{example}


\section{The minimum order and equational graph theories}

In \cite{BucciarelliS04,BucciarelliS08}, Bucciarelli and Salibra defined a notion of ``weak product'' for graph models.
In this paper we prefer to call this construction \emph{gluing} since it does not satisfy the categorical definition of a weak product.

\begin{definition} 
The \emph{gluing} $\Diamond_{k\in K} \lm{G}_k$ of a family $(\gm{G}_k)_{k\in K}$ of graph models with pairwise disjoint webs 
is the graph model freely generated by the partial pair $\sqcup_{k\in K} \cG_k$; its web is denoted by $\Diamond_{k\in K}\cG_k$ instead of $\setcompl{\sqcup_{k\in K}\cG_k}$.
More generally, for any family $(\gm{G}_k)_{k\in K}$ of graph models, $\Diamond_{k\in K}\gm{G}_k$ will denote any gluing of isomorphic copies 
of the $\gm{G}_k$'s with pairwise disjoint webs.
\end{definition}

Note that gluing is commutative and associative up to isomorphism (of graph models).

\begin{lemma}\label{lemma:equivalentproducts} 
Let $(\gm{G}_k)_{k\in K}$, be a family of graph models such that $\gm{G}_k = \gmcompl{\cA_k}$ for some family $(\cA_k)_{k\in K}$ of pairwise disjoint partial pairs.
Then $\Diamond_{k\in K}\gm{G}_k = \gmcompl{\cA}$, where $\cA = \ppcompl{\sqcup_{k\in K}\cA_k}$.
\end{lemma}

\begin{proof} By Remark~\ref{rem:ShouldBeTrue} since, clearly, $\sqcup_{k\in K}\cA_k \subpair \sqcup_{k\in K}\ppcompl{\cA_k}\subpair \ppcompl{\sqcup_{k\in K}\cA_k}$, 
and $\sqcup_{k\in K}\ppcompl{\cA_k} = \sqcup_{k\in K}\cG_k$.
\end{proof}

\begin{proposition}\label{prop:BucciaSali} (Bucciarelli and Salibra \cite[Prop.~2]{BucciarelliS03})\\
Let $(\gm{G}_k)_{k\in K}$ be a family of graph models and $\gm{G} = \Diamond_{k\in K}\gm{G}_k$, then:
\begin{itemize}
\item[(i)]
    $\Int{M}^{\gm{G}_k} = \Int{M}^\gm{G} \cap \setG_k$ for any $M\in\Lambda^o$. Hence:
\item[(ii)]
    $\Thle{\gm{G}}\subseteq \Thle{\gm{G}_k}$, 
\item[(iii)]
    $\Th{\gm{G}}\subseteq \Th{\gm{G}_k}$.
\end{itemize}
\end{proposition}

The existence of a minimum equational graph theory has been shown by Bucciarelli and Salibra in \cite{BucciarelliS04,BucciarelliS08}.
In fact, as observed below, their proof works also for the order theories.

\begin{theorem}\label{thm:minimum} 
There exists a graph model whose order theory is minimum among all order graph theories (hence, the analogue holds for its equational theory).
\end{theorem}

\begin{proof} 
Let $(\cA_k)_{k\in\nat}$ be a family of pairwise disjoint finite partial pairs such that all other finite pairs are isomorphic to at least one $\cA_k$. 
Take $\gm{G} = \Diamond_{k\in K}\gm{G}_k$, where $\gm{G}_k = \gmcompl{\cA_k}$;
by Lemma~\ref{lemma:equivalentproducts}, $\gm{G} = \gmcompl{\cA}$ where $\cA= \sqcup_{k\in\nat} \cA_k$.

We now prove that the order theory, and hence also the equational theory, of $\gm{G}$ is the minimum one. 
Let $e$ be an inequation which fails in some graph model. 
By Corollary~\ref{cor:alphainBandcompletion}$(ii)$ $e$ fails in some $\gmcompl{\cB}$ where $\cB$ is some finite pair, 
hence it fails in some $\gm{G}_k$. 
By Proposition~\ref{prop:BucciaSali}$(ii)$, $e$ fails in $\gm{G}$.
\end{proof}

Recall that the minimum equational graph theory cannot be $\lambda_\beta$ or $\lambda_{\beta\eta}$ by Proposition~\ref{prop:no graph theory eq lambdabeta}.


\section{A L\"owenheim-Skolem theorem for graph models}\label{skolem}

In this section we prove a kind of downwards L\"owenheim-Skolem theorem for graph models: 
every equational/order graph theory is the theory of a graph model having a countable web.
This result positively answers Question 3 in \cite[Sec.~6.3]{Berline00} for the class of graph models.
Note that applying the classical L\"owenheim-Skolem theorem to a graph model $\gm{G}$, viewed as a combinatory algebra $\ca{C}$, would only give a 
countable elementary substructure $\ca{C'}$ of $\ca{C}$.
Such a $\ca{C'}$ does not correspond to any graph model since there exists no countable graph model.
 
Let us first note that the class of total subpairs of a total pair $\cG$ is closed under (finite or infinite) intersections and increasing unions.

\begin{definition}
If $\cA\subpair\cG$ is a partial pair, then the \emph{total subpair of $\cG$ generated by $\cA$}\index{subpair!total -- generated by a pair}
is defined as the intersection of all the total pairs $\cG'$ such that $\cA\subpair\cG'\subpair\cG$. 
\end{definition}

\begin{theorem} {\rm (L\"owenheim-Skolem Theorem for graph models)}\label{thm:Low-Sko} \\
For all graph models $\gm{G}$ there exists a graph model $\gm{G'}$ with a countable web $\cG'\subpair\cG$ such that 
$\Thle{\gm{G'}} = \Thle{\gm{G}}$, and hence such that $\Th{\gm{G'}} = \Th{\gm{G}}$.
\end{theorem}

\begin{proof} 
We will define an increasing sequence of countable subpairs $\cA_n$ of $\cG$, and take for $\cG'$ the 
total subpair of $\cG$ generated by $\cA = \sup_{n\in\nat} \cA_n$.

We start defining $\cA_0$. Let $I$ be the countable set of inequalities between closed $\lambda$-terms which fail in $\gm{G}$.
Let $e\in I$. 
By Corollary~\ref{cor:alphainBandcompletion}$(i)$ there exists a finite partial pair $\cA_e\subpair \cG$ such that 
$e$ fails in every partial pair $\cB$ satisfying  $\cA_e\subpair \cB\subpair\cG$. 
Then we define $\cA_0 = \sup_{e\in I} \cA_e \subpair \cG$.
Assume now that $\cA_{n}$ has been defined, and we define $\cA_{n+1}$ as follows.
Let $\gm{G}'_n$ be the graph model whose web $\cG'_n$ is the total subpair of $\cG$ generated by $\cA_n$.
For each inequality $e = M \sqle N$ which holds in $\gm{G}$ and fails in $\gm{G}'_n$, we consider the set
$L_{e} = \{ \alpha \in\setG'_{n} : \alpha\in \Int{M}^{\gm{G}'_n} \setminus \Int{N}^{\gm{G}'_n}\}$. 
Let $\alpha\in L_{e}$.
Since $\cG'_n\subpair \cG$ and $\alpha\in \Int{M}^{\gm{G}'_n}$, then by Lemma~\ref{lemma:onenvironments} we have that
$\alpha \in  \Int{M}^{\gm{G}}$. By $\Int{M}^\gm{G} \subseteq\Int{N}^\gm{G}$ we also obtain $\alpha \in \Int{N}^{\gm{G}}$.
By Lemma~\ref{lemma:existenceoffinitesubpairs} there exists a partial pair $\cC_{\alpha,e}\subpair \cG$ 
such that $\alpha\in \Int{N}^{\cC_{\alpha,e}}$.
We define $\cA_{n+1}$ as the union of the partial pair $\cA_n$
and the partial pairs $\cC_{\alpha, e}$  for every $\alpha\in L_{e}$. 

As announced, we take for $\cG'$ the total subpair of $\cG$ generated by $\cA = \sup_{n\in\nat}\cA_n$. 
By construction we have, for every inequality $e$ which fails in $\gm{G}$:
$\cA_{e}\subpair\cG'_{n}\subpair\cG'\subpair\cG$. 
Now, $\Thle{\gm{G}'}\subseteq \Thle{\gm{G}}$ follows from Corollary~\ref{cor:alphainBandcompletion}$(1)$ and from the choice of $\cA_e$.

Suppose now, by contradiction, that there exists an inequality $M\ \sqle\ N$ which fails in $\gm{G}'$ but not in $\gm{G}$.
Then there is an $\alpha\in \Int{M}^{\gm{G}'}\setminus \Int{N}^{\gm{G}'}$. By Corollary~\ref{cor:alphainBandcompletion}$(i)$
there is a finite partial pair $\cB\subpair\cG'$ satisfying the following condition:
 for every partial pair $\cC$ such that $\cB\subpair\cC\subpair \cG'$,
we have $\alpha\in \Int{M}^{\cC} \setminus \Int{N}^{\cC}$. 
Since $\cB$ is finite, we have that $\cB\subpair \cG'_n$ for some $n$.
This implies that $\alpha\in \Int{M}^{\gm{G}'_n}\setminus \Int{N}^{\gm{G}'_n}$.
By construction of $\gm{G}'_{n+1}$ we have that $\alpha \in \Int{N}^{\gm{G}'_{n+1}}$; this implies 
$\alpha \in \Int{N}^{\gm{G}'}$. Contradiction.
\end{proof}

As announced at the end of Sections~\ref{sec:methodology}/\ref{sec:Mainresandconj} an alternative and more conceptual proof of Theorem~\ref{thm:Low-Sko}
could be given which can much more easily and transparently be adapted to the other classes of webbed models.

\part{Effective $\lambda$-models}

\section{Effective $\lambda$-models in Scott-continuous semantics}\label{sec:effmodels}

In this section we recall the definition of \emph{effective domains},
also called in the literature ``effectively given domains''.
Then, we introduce the new notion of \emph{effective $\lambda$-models} and \emph{weakly effective $\lambda$-models} and
prove some properties of these models using methods of recursion theory.
In particular we prove that: $(i)$ The equational theory of an effective $\lambda$-model cannot be $\lambda_{\beta}$ or $\lambda_{\beta\eta}$ (Corollary~\ref{cor:nolambdabeta})
$(ii)$ The order theory of an effective $\lambda$-model cannot be r.e.\ (Corollary~\ref{cor:M eff imp Thle M non r.e.}).

\subsection{Effective Scott domains}\label{subsec:effmodels}

All the material developed in this subsection can be found in \cite[Ch.~10]{Viggo94};
its adaptation to DI-domains and DI-domains with coherences can be found in \cite{Gruchalski96}. 

\subsubsection{The category $\bold{ED}$ of effective Scott domains and continuous functions}

\begin{definition}\label{def:eff_dom}
A triple $\D = (\D,\sqle_\D,d)$ is called an \emph{effective domain} if $(\D,\sqle_\D)$ is a Scott domain and 
$d:\nat\to\compel{\D}$ is a numeration of $\compel{\D}$ such that:
\begin{itemize}
\item[(i)]
    the relation ``$d_m$ and $d_n$ have an upper bound'' is decidable in $(m,n)$,
\item[(ii)]
    the relation ``$d_n = d_m \sqcup d_k$'' is decidable in $(m,n,k)$.
\end{itemize}
\end{definition}

It is equivalent to replace $(ii)$ by $(ii)$': the join operator restricted to pairs of compact elements is total recursive 
and ``$d_n = d_m$'' (or, equivalently, ``$d_m\sqle_\D d_n$'') is decidable in $(m,n)$.
The equivalence holds because ``$d_m\sqle_\D d_n$'' is equivalent to ``$d_n = d_m\sqcup d_n$'', and because ``$d_m = d_n$'' is equivalent to ``$d_m\sqle_\D d_n$ and $d_n\sqle_\D d_m$''.

As usual, when there is no ambiguity, we denote by $\D$ the effective domain $(\D,\sqle_\D,d)$.
\begin{notation}
For all $v\in\D$, we set $\Kdn{v} = \{n\st  d_n\sqle_\D v\}$.
\end{notation}

\begin{definition}\label{def:domcomp} An element $v$ of an effective domain $\D$ is called \emph{r.e.} 
(resp. \emph{decidable}) if the set $\Kdn{v}$ is r.e.\ (resp. decidable). 
\end{definition}

In the literature \emph{r.e.\ elements} (of domains) are called ``computable elements'', while our decidable elements were apparently not addressed.
We choose the alternative terminology of r.e.\ elements for the following two reasons: (1) it is more coherent with the usual terminology for elements of 
$\pow{\nat}$ (see Example~\ref{ex:key-example-of-effective-domains}); 
(2) it emphasizes the difference between r.e.\ elements and decidable elements of $\D$.

\begin{notation} $\domre{\D}$ (resp. $\domdec{\D}$) denotes the set of r.e.\ (resp. decidable) elements
of the effective domain $\D$.
\end{notation}

Note that $\compel{\D}
\subseteq \domre{\D}$ and that, in general, $\domre{\D}$ is not a cpo.

\begin{example}
Given an effective numeration of a countable set $D$, the flat domain $D_\bot$ is effective and all its elements are decidable since 
they are compact. 
In particular $\Lambda_\bot$ is effective for the bijective numeration $(-)_\lambda$ defined in the beginning of Section \ref{sec:Recursion in lambda-calculus}. 
\end{example}


\begin{definition} $\bold{ED}$\label{intro:ED} is the category with effective domains as objects, and all continuous functions as morphisms.
\end{definition}

$\bold{ED}$ is a full subcategory of the category of Scott domains; it is Cartesian closed since, 
if $\D,\D'$ are effective domains, also $\D\times \D'$ and $[\D\to \D']$ are effective domains. 
The reader can easily check these facts by himself or find the proofs in \cite{Escardo96}. 

\begin{remark}\label{rem:comp} It is clear that the composition of r.e.\ functions is an r.e.\ function,
moreover it is straightforward to check that the maps \emph{curry} and \emph{eval}, with the usual behaviour, 
and the composition operator $C(f,g) = g\comp f$, are r.e.\ at all types.
Hence, by Theorem \ref{thm:restrcomp}, their restrictions to r.e.\ elements are computable. 
\end{remark}


\subsubsection{Characterizations of r.e.\ continuous functions}

The next proposition gives two other characterizations of r.e.\ functions. 

\begin{proposition}\label{prop:dom-refun}
Let $f:\D\to \D'$ be a continuous function where $\D,\D'$ are effective domains for $d,d'$. 
The following conditions are equivalent:
\begin{alphalist}
\item
    $f\in\domre{[\D\to\D']}$, 
\item
    the relation $d'_m\sqle_{\D'} f(d_n)$ is r.e.\ in $(m,n)$,
\item
    $\{(m,n) \st (d_m,d'_n)\in Tr(f)\}$ is r.e.
\end{alphalist}
and the same holds when ``decidable'' replaces ``r.e.''.
\end{proposition}

We refer to \cite[Ch.~10,~Prop.~3.7]{Viggo94} for a proof of this proposition in the r.e.\ case; of course $(c)$ is just a reformulation of $(b)$. 

\subsubsection{Adequate numerations of $\domre{\D}$}

\begin{definition} A natural number $n\in\nat$ \emph{represents} $v\in\domre{\D}$ if $\W_n = \Kdn{v}$.
\end{definition}

The surjection $\ecomp':\nat\to\domre{\D}$ defined by $\ecomp'(n) = v$ if and only if $n$ represents $v$ is not (in general) a numeration since it 
can be partial. This partiality would create technical difficulties.
However, using standard techniques of recursion theory, it is not difficult to get in a uniform way 
a \emph{total} numeration $\ecomp^\D$ of $\domre{\D}$ \cite[Ch.~10,~Thm.~4.4]{Viggo94}.
In the sequel, we will need some further constraints on $\ecomp^\D$, whose satisfiability is guaranteed by the following proposition.

\begin{proposition}\label{prop:adequate-numeration} There exists a total numeration $\ecomp^\D:\nat\to\domre{\D}$ such that:
\begin{itemize}
\item[(i)]
    $d_n \sqle_\D \ecomp^\D_m$ is r.e.\ in $(m,n)$.
\item[(ii)]
    The inclusion mapping $\iota:\compel{\D}\to\domre{\D}$ is computable with respect to $d,\ecomp^{\D}$.
\end{itemize}
\end{proposition}

\begin{definition} A numeration $\ecomp^\D$ of $\domre{\D}$ is called \emph{adequate} if it fulfills the conditions $(i)$ and $(ii)$ of 
Proposition~\ref{prop:adequate-numeration}.
\end{definition}

\begin{lemma}\label{lemma:universal-numerations}
For all adequate $\ecomp^{\D},\ecomp'^{\D}$ there is a total recursive function $\vphi:\nat\to\nat$ such that $\ecomp^{\D} = \ecomp'^{\D}\comp \vphi$.
\end{lemma}

Hereafter we will always suppose that $\ecomp^\D$ is an adequate numeration of $\domre{\D}$.

Given two effective domains $(\D,\sqle_\D,d)$ and $(\D',\sqle_{\D'},d')$ it is essentially straightforward to obtain, 
in a canonical way, a numeration $\gv_{(d,d')}$ of the compact elements of $[\D\to \D']$ (see \cite[Ch.~10,~Thm.~3.6]{Viggo94})
which is then used to give $[\D\to \D']$ a structure of effective domain.

\begin{theorem}\label{thm:restrcomp} Let $(\D,\sqle_\D,d)$ and $(\D',\sqle_{\D'},d')$ be effective domains, 
then a continuous function $f:\D\to\D'$ is r.e. (with respect to $\gv_{(d,d')}$) if, and only if, its restriction 
$f\restr:\domre{\D}\to\domre{\D'}$ is computable with respect to $\ecomp^\D,\ecomp^{\D'}$.
\end{theorem}

\begin{proof} Relatively easy (the details are worked out in \cite[Ch.~10,~Prop.~4.14]{Viggo94}).
Note that the right handside of the equivalence only depends on $d,d'$ because of Lemma~\ref{lemma:universal-numerations}.
\end{proof}

In particular the previous theorem states that r.e.\ functions preserve the r.e.\ elements.

\begin{example}\label{ex:key-example-of-effective-domains} 
The key example of an effective domain is $(\pow{\nat},\subseteq,d)$ where $d$ is some standard bijective numeration 
$d:\nat\to\nat^*$ of the finite subsets of $\nat$. Here the r.e.\ (resp. decidable) elements are the r.e.\ 
(resp. decidable) sets and the usual map $n\mapsto \W_n$ is an adequate numeration of the r.e.\ elements. 
From Proposition \ref{prop:dom-refun} it follows that a Scott continuous function 
$f: \pow{\nat}\to \pow{\nat}$ is r.e.\ if, and only if, its trace $Tr(f) = \{(a,n) \st a\in\nat^*\textrm{ and }\ n\in f(a)\}$
is an r.e.\ set.
\end{example}

\subsection{Completely co-r.e.\ subsets of $\domre{\D}$}\label{subsec:Visseroneffectivedomains}

Our aim is to infer properties of weakly effective $\lambda$-models using methods of recursion theory.
For this purpose, given an effective domain $\D = (\D,\sqle_\D,d)$ and an adequate numeration $\ecomp^\D:\nat\to\domre{\D}$
we study the properties of the completely co-r.e.\ subsets of $\domre{\D}$.
The work done here could also be easily adapted to DI-domains and DI-domains with coherences. 

\begin{definition}
$A\subseteq \domre{\D}$ is called \emph{completely r.e.} if $A$ is r.e.\ with respect to $\ecomp^\D$;
it is called \emph{trivial} if $A=\emptyset$ or $A = \domre{\D}$.
In a similar way we define \emph{completely co-r.e.\ sets} and \emph{completely decidable sets}.
\end{definition}

This terminology (i.e., the use of ``completely r.e.'' where one would expect ``r.e.'') is coherent with the terminology classically used in recursion theory (see, e.g., \cite{Odifreddi89}).
We will see in Corollary~\ref{cor:dnclosed}$(iii)$ below that there exist no non-trivial completely decidable sets.

\begin{notation} We write $\cocore{\D}$ for the set of completely co-r.e. subsets of $\domre{\D}$.
\end{notation}

\begin{remark}\label{rem:botco-r.e.} $\{\bot_\D\}\in\cocore{\D}$ (therefore $\{\bot_\D\}$ is not completely decidable).
\end{remark}

\begin{theorem} The family of all non-empty completely co-r.e.\ subsets of $\domre{\D}$ has the FIP.
\end{theorem}

\begin{proof}
This follows from Corollary \ref{cor:dnclosed}$(ii)$ below which is itself a consequence of the following extension of the Rice-Myhill-Shepherdson 
theorem (see \cite[Thm.~10.5.2]{Odifreddi89}).
\end{proof}

\begin{theorem}\label{thm:Rice} Let $\D$ be an effective domain and let 
$A\subseteq \domre{\D}$, then $A$ is completely r.e.\ if, and only if, there is an r.e.\ set $E\subseteq \nat$ such that 
$A = \{v \in\domre{\D} \st \exists n\in E\ (\ecomp^\D_n\in\compel{\D}\textrm{ and}$ $ \ecomp^\D_n\sqle_\D v)\}$.
\end{theorem}

\begin{proof} See \cite[Thm.~5.2]{Viggo94}.
\end{proof}

\begin{corollary}\label{cor:dnclosed} With respect to the partial order $\sqle_\D$:
\begin{itemize}
\item[(i)]
    completely r.e.\ sets are upward closed (in $\domre{\D}$),
\item[(ii)]
    completely co-r.e.\ sets are downward closed (in $\domre{\D}$),
\item[(iii)]
    completely decidable sets are trivial.
\end{itemize}
\end{corollary}

\subsection{Weakly effective $\lambda$-models}

In this section we will consider $\lambda$-models living in the Scott semantics; but analogous notions can be defined for 
the stable and the strongly stable semantics.
The following definition of weakly effective $\lambda$-models is completely natural in this context however, in order to obtain stronger results,
we will need a slightly more powerful notion. That is the reason why we only speak of ``weak effectivity'' here.

\begin{definition}\label{def:efflambdamod} A $\lambda$-model $\lm{M} = (\D,\App,\Abs)$ is a 
\emph{weakly effective $\lambda$-model} if: 
\begin{itemize}
\item[(i)]
    $\lm{M}$ is a reflexive object in the category $\bold{ED}$,
\item[(ii)]
    $\App$ and $\Abs$ are r.e.
\end{itemize}
\end{definition}

For the stable and strongly stable semantics, we take respectively: $\bold{EDID}$\label{intro:EDID}, 
the category having effective DI-domains as objects and stable functions as morphisms;
$\bold{EDID^{coh}}$\label{intro:EDIDcoh}, the category having effective DI-domains with coherences as objects 
and strongly stable functions as morphisms.

\begin{remark} Let $\bold{ED^{r.e.}}$ be the subcategory of $\bold{ED}$ with the same objects as $\bold{ED}$ (and the same exponential objects)
but r.e.\ continuous functions as morphisms. 
Using Remark~\ref{rem:comp} it is easy to check that $\bold{ED^{r.e.}}$ inherits the structure of ccc from $\bold{ED}$.
The weakly effective models of Definition~\ref{def:efflambdamod} above are exactly the reflexive objects of $\bold{ED^{r.e.}}$.
We prefer to use the category $\bold{ED}$ first because we think that it is more coherent with the definition of the exponential objects 
to take all continuous functions as morphisms, and second to put in major evidence the only effectiveness conditions which are required.
\end{remark}

We recall a consequence of Theorem \ref{thm:restrcomp} that will be often used later on.

\begin{remark}\label{rem:preserves} Let $\lm{M}$ be weakly effective, then:
\begin{itemize}
\item[(i)]
    If $u,v\in\domre{\D}$ then $\App(u)(v)\in\domre{\D}$,
\item[(ii)]
    If $f\in\domre{[\D\to\D]}$, then $\Abs(f)\in\domre{\D}$.
\end{itemize}
\end{remark}

\begin{convention} In the rest of this section it is understood that we are speaking of a fixed $\lambda$-model 
$\lm{M} = (\D,\App,\Abs)$, where $\D=(\D,\sqle_\D,d)$ is an effective domain and that $\cT= \Th{\lm{M}}$.

Furthermore, we fix a bijective map $\gv_{Var}$ from $\nat$ to the set $Var$ of variables of $\lambda$-calculus.
This gives $Env_\D$ a structure of effective domain. 
\end{convention}

\begin{proposition}\label{prop:subalgebras} 
If $\lm{M}$ is weakly effective, then $(\domre{\D},\bullet,\Int{\bold{K}},\Int{\bold{S}})$ is a combinatory subalgebra of $(\D,\bullet,\Int{\bold{K}},\Int{\bold{S}})$.
\end{proposition}

\begin{proof} 
It follows from Remark~\ref{rem:preserves} that $\Int{\bold{K}}\in\domre{\D}$ and that $\domre{\D}$ is closed under $\bullet$.
The fact that $\Int{\bold{S}}\in\domre{\D}$ is a consequence of the next result.
\end{proof}

\begin{theorem}\label{thm:domreintclosed} 
If $\lm{M}$ is weakly effective, then $\Int{M}\in\domre{\D}$ for all $M\in\Lambda^o_\bot$.
\end{theorem}

\begin{proof} The theorem is a consequence of the following proposition.
\end{proof}

\begin{proposition}\label{prop:domreint} 
If $\lm{M}$ is weakly effective then, for all $M\in\Lambda_\bot$, the function $\Int{M}:Env_\D\to \D$ is r.e.
\end{proposition}

\begin{proof} If $M\equiv \bot$ then $\Int{\bot}$ is the constant function mapping $\rho$ to $\bot$ which is obviously r.e. 
Otherwise, the proof is by structural induction over $M$. 

If $M\equiv x$ then $\Int{M}$ is the map $\rho\mapsto \rho(x)$, i.e., the evaluation of the environment $\rho$ 
on the variable $x$.
It is easy to check that this function is r.e.

If $M\equiv NP$ then $\Int{M} = eval \comp (\App \comp \Int{N},\Int{P})$.
By induction hypothesis and Remark~\ref{rem:comp}, $\Int{M}$ is a composition of r.e.\ functions, hence it is 
r.e.

If $M\equiv \lambda x.N$ then $\Int{M} = \Abs \comp C \comp (curry(f_x), k)$, where $f_x$ is the function 
$(\rho,d)\mapsto \rho[x:=d]$ and $k$ is the constant function mapping $\rho$ to $\Int{N}$. 
We note that $f_x$ is r.e.\ because its restriction $f'_x$ to $\domre{Env_\D}\times\domre{\D}$
is computable. Indeed $f'_x(\rho,d)$ differs from the r.e.\ environment $\rho$ only on $x$ 
where it takes as value the r.e.\ element $d$.
Then this case follows again from induction hypothesis and Remark~\ref{rem:comp}.
\end{proof}

\begin{theorem}\label{thm:2arymap} If $\lm{M}$ is weakly effective, then the function 
$\theta:\Lambda_\bot\times Env_\D\to\D$ defined by $\theta(M,\rho)=\Int{M}_\rho$ is r.e.
\end{theorem}

\begin{proof} 
Since the function $\Int{M}$ is r.e.\ for every $M$, we have that $\Int{M}_\rho\in\domre{\D}$ for all $\lambda$-terms $M$ and 
all r.e.\ environments $\rho$.
Moreover, whenever $M\in\Lambda$ and $\rho\in\domre{Env_\D}$, the proof of Proposition~\ref{prop:domreint} gives an 
effective algorithm to compute in a uniform way the code of $\Int{M}_\rho$ starting from the codes of $M$ and $\rho$.
\end{proof}

\begin{corollary}\label{cor:intcomp}
If $\lm{M}$ is weakly effective and $\rho\in\domre{Env_\D}$, then the function $\Int{-}_{\rho}:\Lambda_\bot\to\D$ is r.e.\ and its restriction to $\Lambda$
is computable with respect to $(-)_\lambda,\ecomp^\D$.
\end{corollary}





\begin{corollary}\label{cor:Vco-r.e.} If $\lm{M}$ is weakly effective and $V\subseteq\domre{\D}$ is completely co-r.e.\ then $\{M\in\Lambda^o \st \Int{M}\in V\}$
is $\beta$-co-r.e.
\end{corollary}

\begin{proof}
Let $\rho\in\domre{Env_\D}$. By Corollary~\ref{cor:intcomp} there exists a recursive map $\vphi_\rho:\nat\to\nat$ tracking the interpretation function 
$M\mapsto \Int{M}_\rho$ of $\lambda$-terms from $\Lambda$ into $\domre{\D}$.
By Remark~\ref{rem:comp1} and since the set $E = \{n\st \ecomp^\D_n \in V\}$ is co-r.e.\ it follows that $\inv{\vphi_\rho}(E)=\{M_\omega\st \Int{M}_\rho \in V\}$, is also co-r.e.
We get the conclusion because $\Lambda^o$ is a decidable subset of $\Lambda$.
\end{proof}

\begin{notation} Given an ordered $\lambda$-model $\lm{M}$ and $\cT= \Th{\lm{M}}$ we set:
\begin{itemize}
\item[(i)]
    $O_{M}= \{ N\in\Lambda^o \st\Int{N} \sqle_\lm{M} \Int{M} \}$, for all $M\in\Lambda^o(\D)$. 
    Note that 
    $O_M$ is a union of $\cT$-classes and that $M\in O_M$.
\item[(ii)]
    $O_\bot = \{ N\in\Lambda^o \st \Int{N} = \bot_\D \}$. 
    Note that either $O_\bot = \emptyset$ or $O_\bot$ consists of a single $\cT$-class,
\item[(iii)]
    $O_\bot^{\omega}= \cup_{n\in\nat}O^n_{\bot}$ where $O^n_{\bot}$ is $O_L$ for 
    $L\equiv\lambda x_1\ldots x_n.\bot_\D$.
\end{itemize}
\end{notation}

Thus, $O_{\bot}$ is just a simplified notation for $O_{\bot}^0$.
The close notations for $O_\bot$ and $\D_\bot$ will not be confusing.

\begin{lemma}\label{lemma:Obot} For all ordered models $\lm{M}$, and $M,N\in\Lambda^o$ we have:
\begin{itemize}
\item[(i)]
    $O_\bot\cup [M]_\cT\subseteq O_M\subseteq \Tins{M}$,
\item[(ii)]
    if $M,N$ are non equivalent hnf's then $O_M\cap O_N\subseteq \Unsolvable$.
\end{itemize}
\end{lemma}

\begin{proof} 
$(i)$ The first inclusion is obvious. The second one follows from the observation that the interpretations of 
two separable $\lambda$-terms are incomparable in any non-trivial partially ordered $\lambda$-model.\\
$(ii)$ is then immediate, once noted that no solvable $\lambda$-term belongs to $M^{ins}\cap N^{ins}$, since an hnf in the intersection 
should be simultaneously equivalent to $M$ and $N$.
\end{proof}

Note that under the hypothesis of Lemma~\ref{lemma:Obot} it can be true that $O_M\cap O_N = \emptyset$ for all $M,N$ which are not $\beta\eta$-equivalent
as shows the model of Di Gianantonio et Al. \cite{DiGianantonioHP95}.

It is interesting to note the following related result, which holds only for graph models and which, as the preceeding one, does not need any 
hypothesis of effectivity.

\begin{lemma}
For all graph models $\gm{G}$, if $N\in O_{\bold{I}}$ then either $N=_{\lambda_{\beta}}\bold{I}$ or $N$ is unsolvable.
\end{lemma}

\begin{proof} Suppose that $N\in O_{\bold{I}}$ and $N$ is solvable. 
Without loss of generality $N$ is an hnf equivalent to $\bold{I}$, hence of the form 
$N\equiv\lambda x.\lambda\vec{z}.x\vec{N},$ with $\vec{z}$ and $\vec{N}$ of the same length $k\geq 0$. 
If $k\geq 1$ it is easy to check that $\gamma= \i{\cG}(\{\i{\cG}(\emptyset^k,\alpha)\},\i{\cG}(\{\alpha\},\i{\cG}(\emptyset^{k-1},\alpha)))\in\Int{N}^{\gm{G}}\setminus\Int{\bold{I}}^\gm{G}$,
where $\i{\cG}(\emptyset^n,\alpha)$ is a shorthand for $\i{\cG}(\emptyset,\i{\cG}(\emptyset\ldots,\i{\cG}(\emptyset,\alpha)\ldots))$.
\end{proof}

Note that this lemma is false for the other classes of models, whether living in the continuous, stable and strongly stable semantics, 
which have been introduced in the literature since they all contain extensional models
(we will give later on more details on these classes).
Looking at the proof we can observe that the reasons why it does not work differ according to the semantics: 
$\gamma\in\Int{\bold{I}}$ in the case of $K$-models, and $\gamma\notin\Int{N}$ in the stable and strongly stable case
(because the injective function $i$ of the web is defined via $Tr_{s})$.

\begin{proposition}\label{prop:botempty1} If $\lm{M}$ is weakly effective and $\cT= \Th{\lm{M}}$, then:
\begin{itemize}
\item[(i)]
    $O_\bot\subseteq\Unsolvable$,
\item[(ii)]
    $O_\bot$ is $\cT$-co-r.e.
\end{itemize}
\end{proposition}

\begin{proof}
$(i)$ is true by Lemma~\ref{lemma:Obot}$(ii)$ since $O_\bot\subseteq O_M\cap O_N$ for all $M,N$.\\
$(ii)$ follows from Remark~\ref{rem:botco-r.e.} and Corollary~\ref{cor:Vco-r.e.}.
\end{proof}

Any sensible model satisfies $\Unsolvable = [\Omega]_\cT\subseteq O_\Omega$. 
Thus, in all sensible models which interpret $\Omega$ by $\bot_\D$ we have $O_{\bot}= O_\Omega = \Unsolvable$ 
(this is the case for example of all sensible graph models). 
On the other hand it is easy to build models satisfying $O_\Omega=\Lambda^o$: for example, finding a graph model $\gm{G}$ with carrier set $G$ such that 
$\Int{\Omega}^\gm{G}=\setG$ is an exercise, 
which also appears as the simplest application of the generalized forcing developed in \cite{BerlineS06}. 
Finally (usual) forcing also allows us to build, for all $M\in\Lambda^o$, a graph model satisfying $\Omega=M$ 
and hence $O_\Omega = O_M$, and 
this is still true for $M\in\Lambda^o(\D)$, and beyond, using generalized forcing.

\begin{proposition}\label{prop:botempty2} 
If $\lm{M}$ is weakly effective and $\cT= \Th{\lm{M}}$ is r.e., then $O_\bot^\omega = \emptyset = O_\bot$.
\end{proposition}

\begin{proof}
Since $O_{\bot}$ consists of zero or one $\cT$-class, it follows from Lemma~\ref{lemma:OinfmoduloT} and Proposition~\ref{prop:botempty1}
that $O_\bot = \emptyset$.
Now it follows, by easy induction on $n$, that $O_{\bot}^n = \emptyset$ for all $n$ since, if $N\in O_\bot^{k+1}$, then
$N\bold{I}\in O_\bot^k$.
\end{proof}

\begin{notation} Given a $\lambda$-model $\lm{M}$ and $\cT = \Th{\lm{M}}$ we set, 
for all $E\subseteq \nat$, $\Intunder{E} =\{N\in\Lambda^o \st \Kdn{\Int{N}}\subseteq E\}$,
where $\Kdn{\Int{N}} = \{ n \st d_n\sqle_\lm{M} \Int{N} \}$. 
Note that $\Intunder{E}$ is a union of $\cT$-classes, which depends on $\lm{M}$ (and not only on $\cT$).
Furthermore, for all $M\in\Lambda^o$, $\Intunder{\Kdn{\Int{M}}} = O_{M}$.
\end{notation}

\begin{theorem}\label{thm:decunderaset}
Let $\lm{M}$ be weakly effective, $\cT= \Th{\lm{M}}$ and $E\subseteq\nat$.
\begin{itemize}
\item[(i)]
    If $E$ is co-r.e.\ then $\Intunder{E}$ is $\cT$-co-r.e.,
\item[(ii)]
    If $E$ is decidable then either $\Intunder{E}=\emptyset$ or $\Intunder{\complement{E}}=\emptyset$.
\end{itemize}
\end{theorem}

\begin{proof} $(i)$ We first note that $E'=\{n \st (\exists m\notin E)\ d_m\sqle_\D \ecomp^\D_n \}$ is r.e.
Hence $\{\ecomp_n \st n\notin E'\}$ is completely co-r.e.
We conclude by Corollary~\ref{cor:Vco-r.e.}.\\
$(ii)$ follows from the FIP, since $\Intunder{E}\cap\Intunder{\complement{E}} = \emptyset$.
\end{proof}

\begin{theorem}\label{Thm:normalprop} 
Let $\lm{M}$ be weakly effective, $\cT= \Th{\lm{M}}$ and $M_1,\ldots,M_n\in\Lambda^o$. If $\Int{M_i}\in\domdec{\D}$ for all $1\le i \le n$, then
$O_{M_1}\cap \cdots \cap O_{M_n}$ is a $\cT$-co-r.e.\ set, which contains a non-empty $\beta$-co-r.e.\ set $\cV$ of unsolvable terms.
\end{theorem}

\begin{proof} 
Since, for all $1\le i\le n$, the set $\Kdn{\Int{M_i}}$ is decidable and $O_{M_i}=\Intunder{\Kdn{\Int{M_i}}}$, then every $O_{M_i}$ is a non-empty $\cT$-co-r.e.\ set by 
Theorem~\ref{thm:decunderaset}$(i)$. 
Hence also $V = O_{M_1}\cap \cdots \cap O_{M_n}$ is a $\cT$-co-r.e.\ set containing $\cV = V\cap \Unsolvable$ which is $\beta$-co-r.e.\ since $\Unsolvable$ is 
$\beta$-co-r.e., and is non-empty by the FIP. 
\end{proof}

\begin{theorem}\label{Thm:normalprop2}
Let $\lm{M}$ be weakly effective and $\cT= \Th{\lm{M}}$. 
If there exists $M\in\Lambda^o$ such that $\Int{M}\in\domdec{\D}$ and $O_{M}\setminus [M]_\cT$ is finite modulo $\cT$, then $\cT$ is not r.e.
\end{theorem}

\begin{proof} Since, by Lemma \ref{lemma:OinfmoduloT}, if $\cT$ is r.e.\ then $O_M/\cT$ is infinite.
\end{proof}

\subsection{Effective $\lambda$-models}\label{def:stronglyefflambdamod}

As proved in Proposition \ref{prop:domreint} weakly effective $\lambda$-models interpret $\lambda$-terms by r.e.\ elements.
The notion of effective $\lambda$-model introduced below has the further key advantage that normal terms are interpreted by
decidable elements and this leads to interesting consequences. 
As we will see in Sections \ref{sec:st&strstareneverre} and \ref{sec:effgraph}, all the models living in the continuous semantics, or in one of its refinements, and introduced \emph{individually}
in the literature are effective.
Furthermore, in the case of webbed models, easy sufficient conditions can be given at the level of the web
in order to guarantee the effectiveness of the $\lambda$-model.


\begin{definition}\label{def:strefflambdamod} A weakly effective $\lambda$-model $\lm{M} = (\D,\App,\Abs)$ is called 
\emph{effective} if it satisfies the following two conditions:
\begin{itemize}
\item[(i)]
    if $d\in\compel{\D}$ and $e_1,\ldots, e_k\in\domdec{\D}$, then $de_1\cdots e_k\in \domdec{\D}$,
\item[(ii)]
    if $f\in\domre{[\D\to\D]}$ and $f(d)\in\domdec{\D}$ for all $d\in\compel{\D}$, then $\Abs(f)\in\domdec{\D}$.
\end{itemize}
\end{definition}

\begin{remark} The key condition is the second one.
Indeed, many $\lambda$-models, and in particular all graph models and all extensional algebraic $\lambda$-models, 
automatically satisfy a property which is stronger than $(i)$, namely $(i')$ below.
\end{remark}

\begin{proposition} Let $\lm{M}=(\D,\App,\Abs)$ be either a graph model, or a $\lambda$-model such that $\D$ is a Scott domain
and $\lm{M}$ satisfies $\Int{\bold{I}}\sqle_\D \Int{\bold{\varepsilon}}$ (equivalently, $id\sqle_{[\D\to\D]} \Abs\comp\App$). Then we have:
\begin{itemize}
\item[(i')] 
    if $d\in\compel{\D}$ and $e_1,\ldots, e_k\in\D$, then $de_1\cdots e_k\in \compel{\D}$.
\end{itemize}
\end{proposition}

\begin{proof} If $\lm{M}$ is a graph model, then the compact elements of $\D$ are exactly the finite subsets of its web.
From this and the definition of application in the case of graph models, it follows that $de$ is compact for every compact $d$ and arbitrary $e$.
If $\lm{M}$ lives in Scott semantics, then $f(e)$ is compact for every compact continuous function $f$ and $e\in\D$.
Indeed, such an $f$ is a finite sup of step functions $\varepsilon_{h,k}$ where $h,k\in\compel{\D}$.
Thus, there only remains to prove that $d\in\compel{\D}$ entails that $\App(d)$ is also compact.
Let $\{f_i \st i\in I\}$ be a non-empty directed set of continuous functions on $\D$, and suppose $\App(d)\sqle\sup_{i\in I} f_i$
then $\Abs(\App(d))\sqle \Abs(\sup_{i\in I} f_i) = \sup_{i\in I}\Abs(f_i)$ by continuity of $\Abs$.
Since $id\sqle \Abs\comp\App$ we get $d\sqle \sup_{i\in I} f_i$ and, by the compactness of $d$, there is a $j\in I$ such that $d\sqle \Abs(f_j)$.
It follows that $\App(d)\sqle\App(\Abs(f_j)) = f_j$ and hence $\App(d)$ is compact.
\end{proof}

\begin{theorem}\label{thm:normal closed decidable} 
If $\lm{M}$ is effective, then for all normal $\lambda$-terms $M\in\Lambda^o$ we have $\Int{M}\in\domdec{\D}$.
\end{theorem}

\begin{proof} 
Since the interpretation of a closed $\lambda$-term is independent of the context, it is enough to show that $\Int{M}_\rho\in\domdec{\D}$
for all normal $M\in\Lambda$ and for all $\rho\in\compel{Env_\D}$. This proof is done by induction over the complexity of $M$. \\
If $M\equiv x$ then $\Int{M}_\rho = \rho(x)$ is a compact element, hence it is decidable.\\
Suppose $M\equiv y N_1\cdots N_k$ with $N_i$ normal for all $1\le i \le k$. 
By definition $\Int{M}_\rho$ is equal to $\Int{y}_{\rho}\cdot\Int{N_1}_{\rho}\cdots \Int{N_k}_{\rho}$.
Hence this case follows from Definition~\ref{def:strefflambdamod}$(i)$, the fact that $\rho(y)$ is compact, 
and the induction hypothesis.\\
If $M\equiv \lambda x.N$ then $\Int{M}_\rho = \Abs(d\mapsto \Int{N}_{\rho[x:= d]})$. Note that, since $\rho\in\compel{Env_\D}$, also
$\rho[x:= d]$ is compact for all $d\in\compel{\D}$. Hence the result follows from the induction hypothesis and 
Definition~\ref{def:strefflambdamod}$(ii)$.
\end{proof}


\begin{corollary}\label{cor:M eff imp Thle M non r.e.} If $\lm{M}$ is effective, then $\Thle{\lm{M}}$ is not r.e.
\end{corollary}

\begin{proof} Let $M\in\Lambda^o$ be normal.
If $\Thle{\lm{M}}$ were r.e., then we could enumerate the set $O_M$.
However, by Theorem~\ref{thm:normal closed decidable} and Theorem~\ref{Thm:normalprop}, this set is co-r.e.\ and it is non-empty because clearly $M\in O_M$.
Hence $O_M$ would be a non-empty decidable set of $\lambda$-terms closed under $\beta$-conversion, i.e., $O_M = \Lambda^o$. 
Since the model is non-trivial and $M$ is arbitrary this lead us to a contradiction.
\end{proof}

\begin{corollary}\label{cor:non-trivialorder} 
If $\lm{M}$ is effective and $\Th{\lm{M}}$ is r.e.\ then $\sqle_\lm{M}$ induces a non-trivial  
partial order on the interpretations of closed $\lambda$-terms.
\end{corollary}

\begin{corollary}\label{cor:nolambdabeta}
If $\lm{M}$ is effective then $\Th{\lm{M}}\neq \lambda_\beta,\lambda_{\beta\eta}$.
\end{corollary}

\begin{proof} By Selinger's result stating that in any partially ordered model whose theory is 
$\lambda_\beta$ or $\lambda_{\beta\eta}$ the interpretations of closed $\lambda$-terms are discretely ordered \cite[Cor.~4]{Selinger96}.
\end{proof}

Recall that in the case of a graph models we know a much stronger result, since we already know from Proposition~\ref{prop:no graph theory eq lambdabeta} that 
for \emph{all} graph models $\gm{G}$ we have $\Th{\gm{G}}\neq\lambda_{\beta}$, $\lambda_{\beta\eta}$. 

\section{Effective stable and strongly stable $\lambda$-models}\label{sec:st&strstareneverre}

There are also many effective models in the stable and strongly stable semantics. 
Indeed, the stable semantics contains a class which is analogous to the class of graph models (see Survey \cite{Berline00}),
namely Girard's class of \emph{reflexive coherent spaces}, called \emph{$G$-models} in \cite{Berline00}.

The material developed in Sections \ref{subsec:compleffpairs} and \ref{subsec:r.e.graphtheories} below for graph models could be adapted 
for $G$-models, even if it is more delicate to complete partial pairs in this case (the free completion process has been described in Kerth
\cite{KerthTh,Kerth01}). This material could also be developed for $H$-models 
(i.e. reflexive hypercoherences; they belong to the strongly stable semantics); 
the free completion process has been worked out in print only for particular $H$-models \cite{GouyTh,BastoneroTh}, but works in greater generality\footnote{
R. Kerth and O. Bastonero, \emph{private communication}, 1997.
}, even though working in the strongly stable semantics certainly adds technical difficulties.

\begin{lemma}\label{lemma:O2containsFcapT} 
If $\lm{M}$ belongs to the stable or strongly stable semantics, then:
$$
    O_{\bold{F}}\cap O_{\bold{T}}\subseteq\{N\st \bold{1} N\in O_{\bot}^{2}\}.
$$
\end{lemma}

\begin{proof}
Suppose $N\in O_{\bold{F}}\cap O_{\bold{T}}$. 
Let $f,g,h\in[\D\to_s\D]$ (resp. $[\D\to_{ss}\D]$) be 
$f=\App(\Int{\bold{T}})$, $g=\App(\Int{\bold{F}})$ and $h=\App(\Int{N})$. 
By monotonicity of $\App$ we have $h\leq_{s}f,g$. 
Now, $g$ is the constant function taking value $\Int{\bold{I}}$, and 
$f(\bot_{\D})= \Int{\lambda y.\bot_{\D}}$. 
The first assertion forces $h$ to be a constant function (Remark~\ref{rem:Berrysproperties}) and the fact
that $h$ is pointwise smaller than $f$ forces $\Abs(h)=\Int{\lambda x.\lambda y.\bot_{\D}}$.
Therefore $\Abs(h)\in O_{\bot}^{2}$. 
It is now enough to notice that, in all $\lambda$-models $\lm{M}=(\D,\App,\Abs)$ we have $\Abs(\App(u))=\Int{\bold{1} u}$ 
for all $u\in\D$ and, in particular, $\Abs(h)=\Int{\bold{1} N}$. Hence $\bold{1} N\in O^2_\bot$.
\end{proof}

\begin{theorem}\label{Thm:8.2} If $\lm{M}$ is effective and belongs to the stable or to the strongly stable semantics then 
$\Th{\lm{M}}$ is not r.e.
\end{theorem}

\begin{proof} By Lemma \ref{lemma:O2containsFcapT} and Proposition~\ref{prop:botempty2}.
\end{proof}

It is easy to check that Lemma \ref{lemma:O2containsFcapT} is false for the continuous semantics. 
We can even give a counter-example in the class of graph models. 
Indeed we know from \cite{BerlineS06} that there exists a graph model $\gm{G}$ (built by forcing) where $\Omega$ acts like intersection (and $\Int{\Omega}^{\gm{G}} = \Int{\bold{1}\Omega}^{\gm{G}}$). 
Equivalently, in $\gm{G}$ we have $\Int{\Omega}^\gm{G} =\Int{\bold{T}}^\gm{G} \cap\Int{\bold{F}}^\gm{G}$ and, hence, 
$\Omega\in (O_{\bold{T}}\cap O_{\bold{F}})\setminus O^2_\bot$. 

\section{Effective graph models}\label{sec:effgraph}

A side effect of this section is to show that effective models are omni-present in the continuous semantics. 
In Section~\ref{subsec:compleffpartpairs} we will introduce a notion of weakly effective (resp. effective) partial pairs,
and in Section~\ref{subsec:compleffpairs} we will prove that they generate weakly effective (resp. effective) $\lambda$-models. 
An analogue of the work done in these two sections could clearly be developed for each of the other classes of webbed models, e.g., 
using the terminology of \cite{Berline00}: $K$-models, pcs-models, filter models (for the continuous semantics),
$G$-models and $H$-models (respectively, for the stable and strongly stable semantics).  
Note that all the $\lambda$-models which have been introduced individually in the literature, to begin with $\gm{P}_\omega$, 
$\gm{E}$ (graph models) and Scott's $\gm{D}^{\infty}$ ($K$-model) are (or could be) presented as generated by webs which happen to be
effective in our sense.


\subsection{Weakly effective and effective pairs}\label{subsec:compleffpartpairs}

\begin{definition}\label{def:effpartialpair} 
A partial pair $\cA$ is \emph{weakly effective} if it is isomorphic to some pair $(E,\ell)$ where $E$ is a decidable subset of $\nat$ 
and $\ell$ is partial recursive with decidable domain. 
It is \emph{effective} if, moreover, $\rg(\ell)$ is decidable.
\end{definition}

\begin{lemma}\label{lemma:efftotpair(N,l)} A total pair $\cG$ is weakly effective if, and only if, it is isomorphic to a total pair $(\nat,\ell)$ 
where $\ell$ is total recursive. It is effective if, moreover, we can choose $\ell$ with a decidable range.
\end{lemma}

\begin{proof} Straightforward.
\end{proof}

Througout this section we suppose, without loss of generality, that all effective and weakly effective pairs have as underlying set a 
subset of $\nat$.

\begin{example} 
$\gm{P}_\omega$, in its original definition (see, e.g., \cite{Bare}) since $\ell$ is defined by $\ell(a,n) = \codepairmix{a}{n}$ 
(with the notation of Section \ref{subsec:recth}); here $\ell$ is also surjective.
\end{example}

\begin{proposition}\label{prop:effpairseffmodels} 
If $\cG$ is a weakly effective (resp. effective) total pair then $\gm{G}$ is a weakly effective (resp. effective) $\lambda$-model.
\end{proposition}

\begin{proof} By Lemma~\ref{lemma:efftotpair(N,l)}, it is enough to prove it for weakly effective pairs of the form $(\nat,\ell)$.
Then it is easy to check, using Definition~\ref{def:graph model}, that $\App^\cG,\Abs^\cG$ are r.e.\ 
Furthermore, condition $(i)$ of Definition~\ref{def:strefflambdamod} (effective $\lambda$-models) is satisfied for all graph models. 
It is finally straightforward to check that condition $(ii)$ holds when $\rg(\ell)$ is decidable. 
\end{proof}

Next, we show that the free completion process preserves the effectivity of the partial pairs.

\subsection{Free completions of (weakly) effective pairs}\label{subsec:compleffpairs}

\begin{theorem}\label{thm:effpartpairseffpairs}
If $\cA$ is weakly effective (resp. effective) then $\ppcompl{\cA}$ is weakly effective (resp. effective).
\end{theorem}

\begin{proof} Suppose $\cA = (A,\j{\cA})$ is a weakly effective partial pair. 
Without loss of generality we can suppose $A = \{ 2^k \st k < \cardinality{A} \}$.
For all $n\in\nat$, we will denote by $\j{n}$ the restriction $\funcompl{\cA}\restr_{A^*_n\times A_n}$ where $A_n$ has been introduced in 
Definition~\ref{def:completionpartialpair}.
We now build $\theta:\setcompl{A}\to\nat$ as an increasing union of functions $\theta_n:A_n\to\nat$ which are defined by induction on $n$.
At each step we set $E_n= \rg(\theta_n)$ and define $\ell_n:E^*_n\times E_n\to E_n$ such that $\theta_n$ is an isomorphism between
$(A_n,\j{n})$ and $(E_n,\ell_n)$. We will take $E = \cup_{n\in\nat} E_n$ and $\ell= \cup_{n\in\nat} \ell_n$.\\
Case $n=0$. We take for $\theta_0 $ the identity on $A$, then $E_0 = A$ and we take $\ell_0=j_\cA$. 
By hypothesis $E_0$ is decidable and $j_\cA$ has a decidable domain and, if $\cA$ is moreover effective, also a decidable range.\\
Case $n+1$. We define
$$
\begin{array}{l}
    \theta_{n+1}(x) = \left\{
    \begin{array}{ll}
	    \theta_n(x) & \textrm{if $x\in A_n$},\\
		p_{n+1}^{\codepairmix{a}{\alpha}} & \textrm{if }x=(a,\alpha)\in (A_{n+1}\setminus A_n).\\
	\end{array} 
    \right.
\end{array}
$$
where $p_{n+1}$ denotes the $(n+1)$-th prime number and $\codepairmix{-}{-}$ is defined in Section~\ref{subsec:recth}. 
Since $A$ and $\dom(j_\cA)$ are decidable by hypothesis and $A_{n}$ is decidable by induction hypothesis then
also $A_{n+1} = A\cup ((A^*_n\times A_n)\setminus \dom(j_\cA))$ is decidable.
$\theta_{n+1}$ is injective, by construction and induction hypothesis.

\noindent Moreover $\theta_{n+1}$ is computable and $E_{n+1}=\rg(\theta_{n+1})$ is decidable since $A_n$ and 
$A_{n+1}\setminus A_n$ are decidable and $\theta_{n}$ and $\codepairmix{-}{-}$ are computable with decidable range.

We define $\ell_{n+1}:E_{n+1}^*\times E_{n+1}\to E_{n+1}$ as follows:
$$
    \ell_{n+1}(a,\alpha) = \left\{
    \begin{array}{ll}
	    \ell_n(a,\alpha) & \textrm{if $q(a,\alpha)\in E_n$},\\
	    q(a,\alpha) & \textrm{if $q(a,\alpha)\in E_{n+1}\setminus E_n$},\\
	\end{array} 
    \right.
$$
where $q = \theta_{n+1} \comp (\inv{\theta}_{n+1},\theta^{-1}_{n+1})$.
The map $\ell_{n+1}$ is partial recursive since $\ell_n$ is partial recursive by induction hypothesis, 
$E_n$ and $E_{n+1}\setminus E_n$ are decidable and $\theta_{n+1}, \img{\theta},\theta^{-1}$ are computable.

It is clear that for all $(a,\alpha)\in A_{n+1}^*\times A_{n+1}$ we have 
$\theta_{n+1}(\j{n+1}(a,\alpha)) \Kleeneq \ell_{n+1}(\img{\theta_{n+1}}(a),\theta_{n+1}(\alpha))$, hence $\theta_{n+1}$ is an isomophism between $(A_{n+1},\j{n+1})$
and $(E_{n+1},\ell_{n+1})$.
Note that, if $\ell_n$ has a decidable range, also $\ell_{n+1}$ has a decidable range.

Then $\theta = \cup_{n\in\nat} \theta_n$ is an isomorphism between $(\setcompl{A},\funcompl{\cA})$ and $(E,\ell)$ where 
$E = \cup_{n\in\nat} \rg(\theta_n)$ and $\ell=\cup_{n\in\nat} \ell_n$.
It is now routine to check that $\theta$ is computable, $E=\rg(\theta)$ is decidable, $\ell:E^*\times E\to E$ is partial recursive,
$\dom(\ell)$ is decidable and, in the case of effectivity, that $\rg(\ell)$ is decidable.
\end{proof}

\begin{corollary} If $\cA$ is weakly effective (resp. effective) then $\gmcompl{\cA}$ is a weakly effective (resp. effective)
graph model.
\end{corollary}

\begin{proof} By Proposition~\ref{prop:effpairseffmodels} and Theorem~\ref{thm:effpartpairseffpairs}.
\end{proof}

\begin{corollary} If $\cA$ is finite, then $\gmcompl{\cA}$ is effective.
\end{corollary}

\begin{lemma}\label{lemma:PeffOPinBase} 
If $\cA$ is weakly effective then $O_{A}$ is $\cT$-co-r.e., where $\cT = \Th{\gmcompl{\cA}}$.
\end{lemma}

\begin{proof} $O_{A}=\Intunder{E}$ for $E=\{ n \st d_{n}\subseteq A \}$. 
If $A$ is decidable then $E$ is decidable, hence $O_{A}$ is co-r.e.\ by Theorem~\ref{thm:decunderaset}, moreover it is obviously $\cT$-closed.
\end{proof}

All the results of this Section would hold for $G$- and $H$- models (even though the corresponding partial pairs and free completion 
process are somewhat more complex than for graph models).

\subsection{Can there be r.e.\ graph theories?}\label{subsec:r.e.graphtheories}

We will now prove several instances of our conjecture(s). We will prove, in particular, that Conjecture \ref{conj:ScottSemantics} 
holds for all free completions of finite partial pairs.
Recall that $O_M= \{ N\in\Lambda^o \st\Int{N} \sqle_\lm{M} \Int{M} \}$ for all $M\in\Lambda^o(\D)$ and that 
the domain associated with a graph model $\gm{G}$ has the form $\D = (\pow{G},\subseteq)$.

\begin{lemma}\label{lemma:omega} 
If $\cA$ is a partial pair, then $\Int{\Omega}^{\gmcompl{\cA}}\subseteq A$, hence $O_A\neq \emptyset$ for the model $\gmcompl{\cA}$.
\end{lemma}

\begin{proof} 
It is well known, and provable in a few lines, that $\alpha\in\Int{\Omega}^{\gmcompl{\cA}}$ implies that $\funcompl{\cA}(a,\alpha)\in a$ for some
$a\in \setcompl{A}^*$ (the details are, for example, worked out in \cite{BerlineS06}). 
Immediate considerations on the rank show that this is possible only if $(a,\alpha)\in \dom(j_\cA)$, which forces $\alpha\in A$.
\end{proof}

\begin{corollary}\label{cor:only unsolvables under Omega}
If $\cA$ is a non total pair and $\Int{U}^{\gmcompl{\cA}}\subseteq \Int{\Omega}^{\gmcompl{\cA}}$, with $U\in\Lambda^o$, then $U$ is unsolvable.
\end{corollary}

\begin{proof} In models $\gmcompl{\cA}$ such that $\cA$ is not total, a solvable $\lambda$-term has an interpretation which contains elements of any rank,
while $\Int{\Omega}^{\gmcompl{\cA}}$ contains only elements of rank $0$. 
\end{proof}

\begin{theorem}\label{thm:extendedthm} 
Let $\cA$ be a weakly effective partial pair. If there exists $E\subseteq A$ such that
$E$ is co-r.e., $O_E\neq \emptyset$ and $E/Aut(\cA)$ is finite (possibly empty), then $\cT = \Th{\gmcompl{\cA}}$ is not r.e.
\end{theorem}

\begin{proof}
We first show that if $\cardinality{E/Aut(\cA)} = k$, for some $k\in\nat$, then $\cardinality{O_E/\cT}\le 2^k$.

Assume $M\in O_E$ and $\alpha\in \Int{M}^{\gmcompl{\cA}}\subseteq E$ then $O(\alpha)$ is included in $\Int{M}^{\gmcompl{\cA}}$ where $O(\alpha)$ is the orbit of $\alpha$ 
in $A$ modulo $Aut(\cA)$.
Indeed if $\theta \in Aut(\cA)$ then $\theta(\alpha) = \bar{\theta}(\alpha) \in \img{\bar{\theta}}(\Int{M}^{\gmcompl{\cA}}) = \Int{M}^{\gmcompl{\cA}}$ 
since $\img{\bar{\theta}}\in Aut(\gmcompl{\cA})$ (Lemma~\ref{lemma:homo-isomorphism}$(ii)$, Theorem~\ref{2.2}$(ii)$).
By hypothesis the number of orbits is $k$; hence the number of all possible interpretations $\Int{M}^{\gmcompl{\cA}}\subseteq E$ cannot overcome $2^k$,
hence $O_E$ is a finite union of $\cT$-classes. 

Since $O_E$ is co-r.e.\ by Theorem~\ref{thm:decunderaset} and $O_E\neq\Lambda^o$, it cannot be decidable; hence $\cT$ cannot be r.e.
\end{proof}



From Theorem~\ref{thm:extendedthm} and Lemma \ref{lemma:omega} we get the following corollaries, 
whose use will be illustrated by the examples after them.

\begin{corollary}\label{cor:5} If $\cA$ is finite, then $\Th{\gmcompl{\cA}}$ is not r.e.
\end{corollary}

\begin{corollary}\label{cor:6} If $\cA$ is weakly effective and $A/Aut(\cA)$ is finite, then $\Th{\gmcompl{\cA}}$ is not r.e.
\end{corollary}

\begin{corollary}\label{cor:7} If $\cA$ is weakly effective and there is a co-r.e.\ set $E$ such that $\Int{\Omega}^{\gmcompl{\cA}}\subseteq E\subseteq A$ and 
$E/Aut(\cA)$ is finite, then $\Th{\gmcompl{\cA}}$ is not r.e.
\end{corollary}

\begin{corollary}\label{cor:8} If $\cA$ is weakly effective, $\Int{\Omega}^{\gmcompl{\cA}}$ is decidable and $\Int{\Omega}^{\gmcompl{\cA}}/Aut(\cA)$ is finite,
then $\Th{\gmcompl{\cA}}$ is not r.e.
\end{corollary}

\begin{corollary}\label{cor:9} If $\cA$ is effective, $\Int{\Omega}^{\gmcompl{\cA}}$ is decidable and 
$\Int{\Omega}^{\gmcompl{\cA}}\cap \Int{N_1}^{\gmcompl{\cA}}\cap\ldots\cap \Int{N_k}^{\gmcompl{\cA}}/Aut(\cA)$
is finite (possibly empty) for some normal terms $N_1,\ldots,N_k\in\Lambda^o$, with $k\in\nat$, then $\Th{\gmcompl{\cA}}$ is not r.e.
\end{corollary}

Let us now give applications of the various corollaries.

\begin{example} Corollary \ref{cor:6} applies to all the usual graph (or webbed) models, indeed: 
\begin{itemize} 
\item[(i)] The Engeler's model $\gm{E}$ is freely generated by $\cA = (A,\emptyset)$, thus all the elements of $A$ play exactly the same role and any 
permutation of $A$ is an automorphism of $\cA$; hence the pair has only one orbit whatever the cardinality of $A$ is.
Of course if $\cA$ is finite, then Corollary~\ref{cor:5} also applies.
\item[(ii)] Idem for the graph-Scott models (including $\gm{P}_\omega$) and the graph-Park models introduced in Example~\ref{exa:Engeler-Pomega}. 
Similarly, the graph model streely generated by $(\{\alpha,\beta\},j)$ where $j(\{\alpha\},\beta)=\beta$ and $j(\{\beta\},\alpha)=\alpha$ only has one orbit.
\item[(iii)] Consider now the mixed-Scott-Park graph models defined in Example~\ref{exa:Engeler-Pomega}$(iv)$.
Then, only the permutations of $A$ which leave $Q$ and $R$ invariant will be automorphisms of $(A,j_\cA)$, and we will have two orbits. 
\end{itemize}
\end{example}

\begin{example} Corollary \ref{cor:8} (and hence Corollary \ref{cor:7}) applies to the following effective pair $\cA$.
$$
\begin{array}{l}
A = \{\alpha_1,\ldots,\alpha_n,\ldots,\beta_1,\ldots,\beta_n,\ldots\}\textrm{ and $j_\cA$ defined by:}\\
j_\cA(\{\beta_n\},\beta_n) = \beta_n, \textrm{ for every }n\in\nat,\\
j_\cA(\{\alpha_1\},\alpha_2) = \alpha_2,\\
j_\cA(\{\alpha_1,\alpha_2\}, \alpha_3) = \alpha_3,\\
\ldots\\
j_\cA(\{\alpha_1,\ldots,\alpha_{n+1}\}, \alpha_{n+2}) = \alpha_{n+2}.\\
\end{array}
$$
Here we have that $\Int{\Omega}^{\gmcompl\cA} = \{\beta_n\st n\in\nat\}$ is decidable and that $\Int{\Omega}^{\gmcompl\cA}/Aut(\cA)$ has cardinality 1,
since every permutation of the $\beta_n$ extends into an automorphism of $\cA$.
Note that the orbits of $A$ are: $\Int{\Omega}$ and all the singletons $\{\alpha_n\}$; in particular $A/Aut(\cA)$ is infinite.
\end{example}

\begin{example}
Corollary \ref{cor:8} applies to the following pair (against the appearance it is an effective pair).
Consider the set $A=\{\beta_{1},\ldots,\beta_{n},\ldots\}$ and the function $j_\cA$ defined as follows:
$j_\cA(\{\beta_n\}, \beta_n) = \beta_n$ if, and only if, $n$ belongs to a non co-r.e.\ set $E\subseteq\nat$.

Then $\Int{\Omega}^{\gmcompl{\cA}}=\{\beta_{n}\st n\in E\}$ consists of only one orbit but is not co-r.e.\ 
However, starting for example from any bijection between $E$ and the set of even numbers, it is easy to find an 
isomorphism of pairs such that $j_\cA$ is partial recursive with decidable range, and hence $\Int{\Omega}^{\gmcompl{\cA}}$ becomes decidable.
\end{example}

\begin{example}
Corollary \ref{cor:9} (and no other corollary) applies to the following effective pair $\cA$.
$$
\begin{array}{l}
A= \{\alpha_{1},\ldots,\alpha_{n},\ldots,\beta_{1},\ldots,\beta_{n},\ldots,\alpha'_{1},\ldots,\alpha'_{n},\ldots,\beta'_{1},\ldots,\beta'_{n},\ldots\},
\textrm{and $j_\cA$ is defined by}:\\
j_\cA(\{\alpha'_{n}\}, \beta_{n})=\alpha'_{n},\\
j_\cA(\{\alpha_{1},\ldots,\alpha_{n}\}, \beta_{n})=\beta'_{n},\\
\end{array}
$$
for all $n\in\nat$.
In this case $\Int{\Omega}^{\gmcompl{\cA}} =\{\beta_{n}\st n\in \nat\}$ is decidable and $\Int{\Omega}^{\gmcompl\cA} \cap \Int{\bold{I}}^{\gmcompl{\cA}} =\emptyset $ 
(note that $\Int{\Omega}^{\gmcompl\cA}/Aut(\cA)$ is infinite).
\end{example}

\begin{example}
(Example of an effective graph model outside the scope of the corollaries of Theorem~\ref{thm:extendedthm}) 
Take for $\cA$ the total pair $\cG = (\nat,\ell)$ where $\ell$ is defined as follows:
$$
    \ell(a,m) = \left\{
    \begin{array}{ll}
	    2k & \textrm{ if $a = \{2k\}$ for $k\in\nat$,}\\
		3^{\finsetenc(a)}5^m & \textrm{otherwise.}\\
	\end{array} 
    \right.
$$
where $\finsetenc:\nat^*\to\nat$ is the effective encoding introduced in Section~\ref{subsec:recth}.
It is easy to check that $\cG$ is effective and that $\Int{\Omega}^{\gm{G}} = \nat$.
Then $O_{\Omega} = \Lambda^o$, hence $\Int{\Omega}^{\gm{G}}/Aut(\cG)$ is infinite. 
\end{example}

Another example of an effective graph model to which the corollaries of Theorem~\ref{thm:extendedthm} are not applicable will be provided by 
Theorem~\ref{thm:effective-minimum}. We do not know whether these two counterexamples could enter in the scope of Theorem~\ref{thm:extendedthm} or not.

\subsection{What about the other classes of webbed models?}

To give a first idea of the strength of Theorem~\ref{thm:extendedthm}, note that all the webbed models that have been introduced individually 
in the literature are (or can be) generated by a weakly effective partial web $W$ \emph{such that $W/Aut(W)$ is finite}.
Of course, the notion of (effective partial) webs, and of automorphism of these webs, should be defined case by case for each class of models.

Now, it should be observed that the results and proofs of Section~\ref{subsec:r.e.graphtheories} hold not only for graph models but also for $G$- and $H$- models.
For Scott continuous semantics the situation is much less clear as soon as we go beyond graph models: the problems already occur at the level of $K$-models
(not to speak of filter models!).

Concerning Lemma~\ref{lemma:omega} 
the difficulty is the following: the web of a $K$-model is a tuple $(D,\preceq,i)$ where $\preceq$ is a preorder on $D$ and $i:D^*\times D\to D$ is
an injection compatible with $\preceq$ in a certain sense.
The elements of the associated reflexive domain are the downward closed subsets of $D$, thus, we should already change the hypothesis for $\Int{\Omega}\subseteq \cldn{A}$, 
where $\cldn{A}$ is the downwards closure of $A$. 
But the real problem is that the control we have on $\Int{\Omega}$ in $K$-models is much looser than in graph models. 
The only thing we know (from Ying Jiang's thesis \cite{JiangTh}) is the following.  
If $\alpha\in\Int{\Omega}$ then there are two sequences 
$\alpha_{n}\in D$ and $a_{n}\in D^*$ such that $\alpha=\alpha_{0}\preceq\alpha_{1}\preceq\ldots\alpha_{n}\preceq\ldots$, 
$\Int{\delta}\supseteq\cldn{a_{0}}\supseteq\cldn{a_{1}} \supseteq\ldots\supseteq\cldn{a_{n}}\supseteq\ldots$ 
and $\beta_n = i(a_{n+1},\alpha_{n+1})\in a_{n}$ for all $n$.
This forces $\beta_{n}$ to be an increasing sequence, included in $\cap_{n\in\nat}(\cldn{a_{n}})$. 
Moreover, if the model is extensional, we have that $\alpha_n = \alpha$ for all $n\in\nat$.
This does not seem to be enough to get an analogue of  Lemma~\ref{lemma:omega}.

Finally, any statement of Theorem~\ref{thm:extendedthm} for $K$-models we should already replace $E\subseteq A$ by $E\subseteq\cldn{A}$ 
to have a chance to have $O_E\neq\emptyset$ (since interpretation of terms are downward closed).

\subsection{An effective graph model having the minimum graph theory}

In this section we show another main theorem of the paper, namely that the minimum order graph theory is the theory of an \emph{effective} graph model. 
As we will see in the next section, this result implies that: $(i)$ no order graph theories can be r.e.; $(ii)$ 
for any closed normal term $M$, there exists a non-empty $\beta$-co-r.e.\ set $\mathscr{V}$ of unsolvable terms 
whose interpretations are below that of $M$ in all graph models.

\begin{theorem}\label{thm:effective-minimum} 
There exists an effective graph model whose order/equational theory is the minimum order/equational graph theory.
\end{theorem}

\begin{proof} 
It is not difficult to define an effective 
numeration $\cN$ of all the finite partial pairs whose carrier set is a subset of $\nat$. 
We now make the carrier sets $N_k$, for $k\in\nat$, pairwise disjoint.
Let $p_k$ be the $k$-th prime number.
Then we define another finite partial pair $\cA_k$ as follows:
$A_k = \{p_k^{x+1}: x\in N_k\}$ and 
$j_{\cA_k}(\{p_k^{\alpha_1+1},\dots,p_k^{\alpha_n+1}\},p_k^{\alpha +1}) = p_k^{j_{\cN_k}(\{\alpha_1,\dots,\alpha_n\},\alpha)+1}$
for all $(\{\alpha_1,\dots,\alpha_n\},\alpha)\in \dom(j_{\cN_k})$.
In this way we get an effective bijective numeration of all the finite partial pairs $\cA_k$.\\
Let us take $\cA = \sqcup_{k\in\nat} \cA_k$.
It is an easy matter to prove that $A$ is a decidable subset of $\nat$ and that $\j{\cA}$ is a computable map 
with decidable domain and range. 
It follows from Theorem~\ref{thm:effpartpairseffpairs} that $\gmcompl{\cA}$ is an effective graph model.

Finally, with the same reasoning done in the proof of Theorem~\ref{thm:minimum}, we can conclude that $\Thle{\gmcompl{\cA}}$ 
(resp. $\Th{\gmcompl{\cA}}$) is the minimum order graph theory (resp. equational graph theory).
\end{proof}

Let $\Tmin$ and $\Omin$ be, respectively, the minimum equational graph theory and the minimum order graph theory.

\begin{proposition} $\Tmin$ and $\Omin$ are in fact the theories of countably many non-isomorphic effective graph models.
\end{proposition}

\begin{proof} Since, in the proof of Theorem~\ref{thm:effective-minimum}, there exist countably many choices for the effective numeration $\cN$
which give rise to non-isomorphic graph models having minimal theory. 
For example, for every recursive sequence $(n_k)_{k\in\nat}$, take a recursive numeration which repeats $n_k$-times the pair $\cN_k$.
\end{proof}

\begin{proposition} $\Tmin$ is an intersection of a countable set of non r.e.\ equational graph theories.
\end{proposition}

\begin{proof}
By the proof of Theorem \ref{thm:minimum} $\Omin = \cap\Thle{\gmcompl{\cA_k}}$ where $\cA_k$ ranges over all finite pairs.
By Corollary~\ref{cor:5} these theories are not r.e.
\end{proof}

\subsection{Applications to the class of all graph models}

The following two results are consequences of Theorem~\ref{thm:effective-minimum}.

\begin{theorem} 
For all graph models $\gm{G}$, $\Thle{\gm{G}}$ is not r.e.
\end{theorem}

\begin{proof} 
Let $\Gmin$ be some fixed effective graph model having minimum order theory and $M$ be any closed normal $\lambda$-term. 
Since $\Gmin$ is effective, Theorem~\ref{thm:normal closed decidable} implies that $\Int{M}^{\Gmin}$ is decidable,
hence $O_M^{\Gmin} = \{ N \in \Lambda^o: \Int{N}^{\Gmin} \subseteq \Int{M}^{\Gmin}\}$ is $\beta$-co-r.e. by Theorem~\ref{Thm:normalprop}.

Suppose, now, that $\gm{G}$ is a graph model such that $\Thle{\gm{G}}$ is r.e.
Then $O^\gm{G}_M = \{ N\in\Lambda^o : \Int{N}^\gm{G}\subseteq \Int{M}^\gm{G}\}$ is a $\beta$-r.e.\ set
which contains the co-r.e. set $O_M^{\Gmin}$. Thus, by the FIP we get $O^\gm{G}_M = \Lambda^o$. 

From the arbitrariness of $M$, it follows that $\bold{F} \in O_{\bold{T}}$, and \emph{vice versa}, 
so we get $\Int{\bold{F}}^\gm{G}=\Int{\bold{T}}^\gm{G}$, contradiction.
\end{proof}

\begin{proposition}\label{uns term under I}
For all normal $M_1,\ldots,M_n\in\Lambda^o$ there exists a non-empty $\beta$-co-r.e.\ set $\cV$ of closed unsolvable terms such that: 
$$
    \textrm{For all graph models $\gm{G}$: }\forall U\in\cV\ (\Int{U}^\gm{G}\subseteq \Int{M_1}^\gm{G}\cap\ldots\cap\Int{M_n}^\gm{G}).
$$
\end{proposition}

\begin{proof} 
Let $\Gmin$ be some fixed effective graph model having minimum order theory.
Since $\Gmin$ is effective and $M_1,\ldots,M_n$ are closed normal $\gl$-terms, Theorem~\ref{thm:normal closed decidable} implies that every $\Int{M_i}^{\Gmin}$ 
is decidable.
Thus, from Theorem~\ref{Thm:normalprop} it follows that there exists a set $\cV$ of unsolvable terms such that for all 
$U\in\cV$ we have $\Int{U}^{\Gmin}\subseteq \Int{M_i}^{\Gmin}$ for all $1\le i\le n$.
Then the conclusion follows since $\Thle{\Gmin}$ is the minimum order graph theory.
\end{proof}

\begin{remark}
The authors do not know any concrete example of an unsolvable term $U$ satisfying the above inclusion; not even of an unsolvable $U$ such that, for all graph model $\gm{G}$,
we have $\Int{U}^\gm{G}\subseteq \Int{\bold{I}}^\gm{G}$.
\end{remark}

Nevertheless, Proposition~\ref{uns term under I} has the following interesting corollary, which intuitively expresses that there are lots of easy terms for which we will never be able to give a semantic proof of their easiness (which by no means implies that we could prove syntactically their easiness), at least using graph models.
Following the terminology of \cite{BerlineS06} we say that $U\in\Lambda^o$ is \emph{graph easy} if for all $M\in\Lambda^o$ there is a graph model $\gm{G}$ satisfying $\Int{U}^\gm{G} = \Int{M}^\gm{G}$. The most known example of graph easy term is of course $\Omega$ by \cite{BaetenB79} and consequently all the $\Omega N$ for $N\in\Lambda^o$.

\begin{corollary} There exist a non-empty $\beta$-co-r.e.\ set $\cV'$ of easy terms which are not graph-easy.
\end{corollary}

\begin{proof} Let $\cV$ be the $\beta$-co-r.e.\ set such that for all $U\in\cV$ and $\Int{U}^\gm{G}\subseteq \Int{\bold{I}}^\gm{G}$ all graph models $\gm{G}$ (using Proposition~\ref{uns term under I}),
then it is sufficient to take $\cV' = \cV\cap\Lambda_{easy}$.
\end{proof}

\bibliographystyle{giulio}
\bibliography{giulio}

\begin{thebibliography}{10}
\expandafter\ifx\csname url\endcsname\relax
  \def\url#1{\texttt{#1}}\fi
\expandafter\ifx\csname urlprefix\endcsname\relax\def\urlprefix{URL }\fi

\bibitem{BaetenB79}
J.~Baeten, B.~Boerboom, Omega can be anything it should not be, in: Proc.\
  Koninklijke Netherlandse Akademie van Wetenschappen, Serie A, Indag.
  Matematicae 41, 1979.

\bibitem{Bare}
H.~P. Barendregt, The Lambda calculus: Its syntax and semantics, North-Holland,
  Amsterdam, 1984.

\bibitem{BarendregtCD83}
H.~P. Barendregt, M.~Coppo, M.~{Dezani-Ciancaglini}, A filter lambda model and
  the completeness of type assignment, Journal of Symbolic Logic 48~(4) (1983)
  931--940.

\bibitem{BastoneroTh}
O.~Bastonero, Mod\`eles fortement stables du $\lambda$-calcul et r\'esultats
  d'incompl\'etude, th\`ese de {D}octorat (1996).

\bibitem{BerardiB02}
S.~Berardi, C.~Berline., $\beta\eta$-complete models for system ${F}$,
  Mathematical Structure in Computer Science 12 (2002) 823--874.

\bibitem{Berline00}
C.~Berline, From computation to foundations via functions and application: The
  $\lambda$-calculus and its webbed models, Theoretical Computer Science 249
  (2000) 81--161.

\bibitem{Berline06}
C.~Berline, Graph models of $\lambda$-calculus at work, and variations, Math.
  Struct. for Comput. Sci. 16 (2006) 1--37.

\bibitem{BerlineMS07}
C.~Berline, G.~Manzonetto, A.~Salibra, Lambda theories of effective lambda
  models, in: 16th EACSL Annual Conference on Computer Science and Logic
  (CSL'07), LNCS, vol. 4646, 2007.

\bibitem{BerlineS06}
C.~Berline, A.~Salibra, Easiness in graph models, Theoretical Computer Science
  354 (2006) 4--23.

\bibitem{Berry78}
G.~Berry, Stable models of typed lambda-calculi, in: In Proceedings of the
  Fifth Colloquium on Automata, Languages and Programming, LNCS 62,
  Springer-Verlag, Berlin, 1978.

\bibitem{BerryTh}
G.~Berry, Mod\`eles compl\`etement ad\'equats et stable des $\lambda$-calculs
  typ\'es, th\`ese de {D}octorat d'\'etat (1979).

\bibitem{BerryC85}
G.~Berry, P.-L. Curien, J.-J. L{\'e}vy, {Full Abstraction for Sequential
  Languages: the State of the Art}, Algebraic Methods in Semantics (1985)
  89--132.

\bibitem{BucciarelliE91}
A.~Bucciarelli, T.~Ehrhard, Sequentiality and strong stability, in: Sixth
  Annual IEEE Symposium on Logic in Computer Science, IEEE Computer Society
  Press, 1991.

\bibitem{BucciarelliS08}
A.~Bucciarelli, A.~Salibra, Graph lambda theories, Mathematical Structures in
  Computer Science (to appear).

\bibitem{BucciarelliS03}
A.~Bucciarelli, A.~Salibra, The minimal graph model of lambda calculus, in:
  MFCS'03, LNCS, Springer-Verlag, 2003.

\bibitem{BucciarelliS04}
A.~Bucciarelli, A.~Salibra, The sensible graph theories of lambda calculus, in:
  19th Annual IEEE Symposium on Logic in Computer Science (LICS'04), IEEE
  Computer Society Publications, 2004.

\bibitem{CoppoD80}
M.~Coppo, M.~Dezani-Ciancaglini, An extension of the basic functionality theory
  for the $\lambda$-calculus, Notre-Dame Journal of Formal Logic 21~(4) (1980)
  685--693.

\bibitem{CoppoDHL82}
M.~Coppo, M.~{Dezani-Ciancaglini}, F.~Honsell, G.~Longo, {Extended Type
  Structures and Filter Lambda Models}, in: G.~Lolli, G.~Longo, A.~Marcja
  (eds.), Logic Colloquium 82, North-Holland, Amsterdam, the Netherlands, 1984.

\bibitem{CoppoDZ87}
M.~Coppo, M.~Dezani-Ciancaglini, M.~Zacchi, Type theories, normal forms and
  ${D}_\infty$ $\lambda$-models, Information and Computation 72 (1987) 85--116.

\bibitem{David01}
R.~David, Computing with {B\"o}hm trees., Fundam. Inform. 45~(1-2) (2001)
  53--77.

\bibitem{DiGianantonioHP95}
P.~{Di Gianantonio}, F.~Honsell, G.~Plotkin, Uncountable limits and the lambda
  calculus, Nordic Journal of Computing 2~(2) (1995) 126--145.

\bibitem{Escardo96}
M.~H. Escard{\'o}, Pcf extended with real numbers., Theoretical Computer
  Science 162~(1) (1996) 79--115.

\bibitem{GianniniL84}
P.~Giannini, G.~Longo, Effectively given domains and lamba-calculus models,
  Information and Control 62 (1984) 36 -- 63.

\bibitem{GouyTh}
X.~Gouy, Etude des th\'eories \'equationnelles et des propri\'et\'es
  alg\'ebriques des mod\`eles stables du $\lambda$-calcul, th\`ese de
  {D}octorat (1995).

\bibitem{Gruchalski96}
A.~Gruchalski, Computability on di-domains., Inf. Comput. 124~(1) (1996) 7--19.

\bibitem{HonsellTLCA}
F.~Honsell, {TLCA} list of open problems: Problem \# 22,
  \newline\url{http://tlca.di.unito.it/opltlca/problem22.pdf} (2007).

\bibitem{HonsellR92}
F.~Honsell, S.~{Ronchi Della Rocca}, An approximation theorem for topological
  lambda models and the topological incompleteness of lambda calculus, Journal
  of Computer and System Sciences 45 (1992) 49--75.

\bibitem{JiangTh}
Y.~Jiang, Consistence et inconsistence de th\'eories de $\lambda$-calcul
  \'etendus, th\`ese (1993).

\bibitem{Kerth94}
R.~Kerth., $2^{\aleph_0}$ mod\`eles de graphes non \'equationnellement
  \'equivalents, in: Notes de comptes-rendus de l'Acad\'emie des Sciences, vol.
  318, 1994, pp. 587--590.

\bibitem{KerthTh}
R.~Kerth, Isomorphism et \'equivalence \'equationnelle entre mod\`eles du
  $\lambda$-calcul, th\`ese (1995).

\bibitem{Kerth98}
R.~Kerth, The interpretation of unsolvable terms in models of pure
  $\lambda$-calculus, J. Symbolic Logic 63 (1998) 1529--1548.

\bibitem{Kerth98b}
R.~Kerth, Isomorphism and equational equivalence of continuous lambda models,
  Studia Logica 61 (1998) 403--415.

\bibitem{Kerth01}
R.~Kerth, On the construction of stable models of $\lambda$-calculus,
  Theoretical Computer Science 269 (2001) 23--46.

\bibitem{Krivine93}
J.-L. Krivine, Lambda-calculus, types and models, Ellis Horwood, New York,
  1993, translated from the ed. Masson, 1990, French original.

\bibitem{Longo83}
G.~Longo, Set-theoretical models of $\lambda $-calculus: theories, expansions,
  isomorphisms, Ann. Pure Appl. Logic 24~(2) (1983) 153--188.

\bibitem{LusinS04}
S.~Lusin, A.~Salibra, The lattice of lambda theories, Journal of Logic and
  Computation 14 (2004) 373--394.

\bibitem{ManzonettoS06}
G.~Manzonetto, A.~Salibra, Boolean algebras for lambda calculus, in: Proc.\
  21th {IEEE} Symposium on Logic in Computer Science ({LICS} 2006), 2006.

\bibitem{Odifreddi89}
P.~Odifreddi, Classical Recursion Theory, Elsevier, North-Holland, Amsterdam,
  1989.

\bibitem{Plotkin77}
G.~D. Plotkin, {LCF} considered as a programming language, Theoretical Computer
  Science 5 (1977) 223--255.

\bibitem{Plotkin93}
G.~D. Plotkin, Set-theoretical and other elementary models of the
  lambda-calculus, Theoretical Computer Science 121~(1\&2) (1993) 351--409.

\bibitem{Ronchi82}
S.~Ronchi Della~Rocca, Characterization theorems for a filter lambda model,
  Information and Control 54 (1982) 201 -- 216.

\bibitem{Salibra03}
A.~Salibra, Topological incompleteness and order incompleteness of the lambda
  calculus. \uppercase{LICS}'01 \uppercase{S}pecial \uppercase{I}ssue, No.~4,
  ACM Transactions on Computational Logic, 2003.

\bibitem{Scott72}
D.~S. Scott, Continuous lattices, in: Toposes, algebraic geometry and logic,
  Springer-Verlag, Berlin, 1972.

\bibitem{Selinger96}
P.~Selinger, Order-incompleteness and finite lambda models, extended abstract,
  in: LICS '96: Proceedings of the 11th Annual IEEE Symposium on Logic in
  Computer Science, IEEE Computer Society, Washington, DC, USA, 1996.

\bibitem{Selinger03}
P.~Selinger, Order-incompleteness and finite lambda reduction models,
  Theoretical Computer Science 309 (2003) 43--63.

\bibitem{Viggo94}
V.~Stoltenberg-Hansen, I.~Lindstr\"om, E.~R. Griffor, Mathematical theory of
  domains, Cambridge University Press, New York, NY, USA, 1994.

\bibitem{Visser80}
A.~Visser, Numerations, $\lambda$-calculus, and arithmetic, in: Hindley, Seldin
  (eds.), Essays on Combinatory Logic, Lambda-Calculus, and Formalism, Academic
  Press, 1980, pp. 259--284.

\end{thebibliography}

\appendix
\newpage

\section{The Visser topology on $\Lambda$}\label{sec:Vissertop}

It could be possible to give a topological flavour to the framework of effective $\lambda$-models by using the Visser topology.
Nevertheless, for obtaining the results which are present in this paper, no real topological manipulation would be performed.
In particular, the topology itself is never used: just its base.

The general definition of the Visser topology on an enumerated set $(X,\gv_X)$ was introduced by Visser in \cite{Visser80}
as well as its application to $\Lambda$.
We recall this definition below. 

\subsection{General definition of the Visser topology}
If $\gamma = (X,\gv_X)$ is a numeration, we write $m \sim_\gamma n$ for $\gv_X(m) = \gv_X(n)$.
A numeration $\gamma$ is called \emph{precomplete} if, for every partial recursive $\vphi$, there is a total recursive 
$\psi$ such that $\forall n\in \dom(\vphi)$ we have $\vphi(n)\sim_\gamma \psi(n)$.
\emph{If $\gamma_1=(X,\gv_X)$ and $\gamma_2=(Y,\gv_Y)$ are numerations}, then $f$ is a \emph{morphism} from 
$\gamma_1$ to $\gamma_2$ if $f$ is a function from $X$ to $Y$ and if there is a total recursive $\vphi$ 
such that the following diagram commutes:
$$
\xymatrix{
\nat \ar[rr]^\vphi \ar[d]_{\gv_X} &&
\nat \ar[d]^{\gv_Y} \\
X \ar[rr]_{f} && Y\\ 
}
$$
The numeration $(-)_\lambda$ introduced in Section~\ref{sec:co-r.e. sets} of $\Lambda$ is precomplete. 

\begin{definition}\label{def:Vissertop} Let $\gamma=(\gv_X,X)$ be a numeration. 
The Visser topology on $X$ is the topology generated by the base 
$$
     \Base_X = \{ O\subseteq X \st \inv{\gv_X}[O]\textrm{ is co-r.e.} \}.
$$
\end{definition}

The fact that $\Base_X$ is a topological base follows from the closure of co-r.e.\ sets under finite intersection. 

\begin{theorem}\label{thm:hyperconnectedness} If the numeration $\gamma=(\gv_X,X)$ is precomplete then the Visser 
topology on $X$ is hyperconnected, i.e., every two non-empty open sets have non-empty intersection.
\end{theorem}

\begin{proof} See \cite[Thm.~2.5]{Visser80}.
\end{proof}

The hyperconnectedness is the most important and useful property of the Visser topology.

\subsection{The Visser topology on $\Lambda$}\label{subsec:lambdavisser}

We provide now the instance of the definition in the context of $\lambda$-calculus.

\begin{definition}
The Visser topology on $\Lambda$ is the topology whose basic open sets are the subsets $O$ such that:
\begin{itemize}
\item
    $O$ is closed under $\beta$-conversion,
\item
    $\{(M)_\omega \st M\in O \}$ is a co-r.e.\ set.
\end{itemize}
\end{definition}

\begin{proposition}\label{prop:hyperconnection} The Visser topology on $\Lambda$ is hyperconnected.
\end{proposition}

\begin{proof} 
The numeration $((-)_\lambda,\Lambda)$ is precomplete by \cite[Thm.~1.6.1.1]{Visser80};
then the result follows from Theorem~\ref{thm:hyperconnectedness}.
\end{proof}








\subsection{Visser-continuity}

The second main property of the Visser topology on $\Lambda$ is the following.

\begin{proposition}\label{prop:Vissercontinuity} Every $\lambda$-definable map on $\Lambda$ is Visser continuous.
\end{proposition}

\begin{proof} The inverse image of every r.e.\ set via a computable function is an r.e.\ set.
\end{proof}

Proposition~\ref{prop:Vissercontinuity} is still true for $n$-ary functions, for the same reason, if we consider the Visser topology $\cW_n$ on $\Lambda^n$. 
But it is important to keep in mind that it is in general false for the product topology $\cV^n$ on $\Lambda^n$.
Note that all non-empty open sets of $\mathcal{V}^{2}$ meet the diagonal
$\Delta = \{ (M,N)\st M=_{\lambda_{\beta}} N \}$ (by the hyperconnectedness of $\mathcal{V})$, hence $\Delta^{c}$, which is obviously 
basic open in $\cW$, is not even open in $\mathcal{V}^{2}$;
this proves $\cV^2 \neq \cW_2$.
The following result is a variation of this remark.

\begin{proposition} The application function on $\Lambda$ is:
\begin{itemize}
\item[(i)] Visser continuous in each coordinate,
\item[(ii)] continuous with respect to $\cW_2$, but
\item[(iii)] not continuous with respect to $\mathcal{V}^{2}$.
\end{itemize}
\end{proposition}

\begin{proof} $(i)$ and $(ii)$ follow from the fact that application is computable.

(iii) Let $\psi$ be defined on $\Lambda^2$ by $\psi(M,N)=$ $\Omega MN$, and let
$\cT$ be the r.e.\ $\lambda$-theory generated by $\Omega xx=\Omega$.
Since $\cT\subseteq \cH$, obviously $\cT$ is consistent. 
Let $V=\{(M,N) \st \Omega MN\neq_{\cT}\Omega\}$; 
Since $\cT$ is r.e., $[\Omega]_\cT$ is r.e.\ and hence its complement $O$ is Visser open (and non-empty). 
If the application were $\mathcal{V}^{2}$-continuous, then $\psi$ would be $\mathcal{V}^{2}$-continuous, 
and $V=\inv{\psi}(O)$ would be $\mathcal{V}^{2}$-open. 
Now, $V$ is non-empty because $(M,N)\in V$ if, and only if, $M \not=_{\cT} N$ 
(Salibra \cite{Salibra03}), hence $V$ would meet the diagonal, which contradicts the definition of $\cT$.
\end{proof}

\section{Sometimes gluings are weak products}\label{app:Sometimes gluings are weak products}

There exist obvious morphisms $\iota_n:\cG_n\to\Diamond_{k\in K}\cG_k$, but, in general there are no ``projection-like'' morphisms. 
There are, however, good cases where projections $\pi_n:\Diamond\gm{G}_n\to \gm{G}_n$ can be found, and moreover $\cG_n\retract \Diamond_{k\in K}\cG_k$ for all $n\in K$.
In these cases we get easily that $\Th{\Diamond_{k\in K}\gm{G}_k}\subseteq \cap_{k\in K} \Th{\gm{G}_k}$, an inclusion which was proved in
full generality in \cite{BucciarelliS08} with a much more intricate proof.

The first example of such a situation is ``self-gluing''. 
We denote by $\gm{G}^k$ the gluing of $\gm{G}$ with itself $k$-times, i.e., $\gm{G}\Diamond\cdots \Diamond\gm{G}$.
The next result is easy to be checked.

\begin{proposition} For all graph models $\gm{G}$ we have $\gm{G}\retract \gm{G}^k$.
\end{proposition}

\begin{definition} 
Given a partial pair $\cA$, an element $\alpha\in A$ is \emph{isolated} if $j_\cA (\{\alpha\},\alpha)=\alpha$ and \emph{critical}
if $(\emptyset,\alpha)\in \dom(j_\cA)$.
\end{definition}

\begin{definition} A partial pair $\cA$ is \emph{good} if it has at least one isolated element and has no critical elements.
A graph model $\gm{G}$ is \emph{good} if there exists a good pair $\cA$ such that $\gm{G} = \gmcompl{\cA}$.
\end{definition}

\begin{remark} 
If $\cA$ is not total then $\ppcompl{\cA}$ contains critical elements, hence is not good, but $\gmcompl{\cA}$ can be.
\end{remark}

For example, the graph-Park models are good, whilst the graph-Scott models (including $\gm{P}_\omega$) are not.
We recall that both classes were introduced in Example~\ref{exa:Engeler-Pomega}.

\begin{remark} The class of good graph models is closed under $\Diamond$ (by Lemma~\ref{lemma:equivalentproducts}).
\end{remark}

\begin{proposition} If $(\cA_k)_{k\in K}$ is a family of pairwise disjoint good partial pairs, then $\gmcompl{\cA_n}\retract\Diamond_{k\in K}\gmcompl{\cA_k}$ for all $n\in K$.
\end{proposition}

\begin{proof} By Lemma~\ref{lemma:equivalentproducts}, we recall that $\Diamond_{k\in K}\gmcompl{\cA_k} = \gmcompl{\cA}$, where $\cA = \ppcompl{\sqcup_{k\in K} \cA_k}$.
For all $k\in K$, let $\alpha_k\in A_k$ be an isolated element.
We can now define, for any $n$, a map $\pi_n:\sqcup_{k\in K} \cA_k\to \cA_n$ as follows: $\pi_n(x) = x$ if $x\in A_n$ and $\pi_n(x)=\alpha_n$, otherwise.
Since the pairs have no critical elements, it is easy to prove that $\pi_n,\iota_n:A_n\retract \sqcup_{k\in K} A_k$
where $\iota_n:A_n\to \sqcup_{k\in K} A_k$ is the inclusion mapping. Hence we conclude with Lemma~\ref{Lemma Extending retractions}.
\end{proof}

\begin{corollary}\label{cor:good theories le} If $(\cG_k)_{k\in K}$ is a family of good graph models, then $\Th{\Diamond_{k\in K}\cG_k}\subseteq \cap_{k\in K} \Th{\cG_k}$.
\end{corollary}

\begin{corollary} There is a minimal theory of good graph model and it is semi-sensible.
\end{corollary}

\begin{proof} Let $(\cA_k)_{k\in K}$ be a family of pairwise disjoint good finite pairs such that all other finite good pairs are isomorphic to
at least one $\cA_k$.
Take $\gm{G} = \Diamond_{k\in K}\gmcompl{\cA_k}$. It is easy to check that $\Th{\gm{G}}$ is minimal and it is 
semi-sensible by Theorem~\ref{thm:no semi-sensible free completions}.
\end{proof}
\end{document}